\documentclass{amsart}%
\usepackage{amssymb}
\usepackage{amsfonts}
\usepackage{geometry}
\usepackage{graphicx}
\usepackage{amsmath}%
\providecommand{\U}[1]{\protect\rule{.1in}{.1in}}
\newtheorem{theorem}{Theorem}
\theoremstyle{plain}

\newtheorem{conjecture}[theorem]{Conjecture}

\newtheorem{definition}[theorem]{Definition}

\newtheorem{notation}[theorem]{Notation}

\newtheorem{proposition}[theorem]{Proposition}
\newtheorem{remark}[theorem]{Remark}

\numberwithin{equation}{section}
\begin{document}
\title[Kakeya conjectures]{Equivalence of linear and trilinear Kakeya conjectures in three dimensions}
\author{Cristian Rios}
\address{University of Calgary\\
Calgary, Alberta, Canada}
\email{crios@ucalgary.ca}
\author{Eric Sawyer}
\address{McMaster University\\
Hamilton, Ontario, Canada}
\email{sawyer@mcmaster.ca}
\thanks{Eric Sawyer's research supported in part by a grant from the National Sciences
and Engineering Research Council of Canada}
\maketitle

\begin{abstract}
We prove the equivalence of two Kakeya conjectures in three dimensions:

\begin{enumerate}
\item The Kakeya maximal operator conjecture

\item The \emph{disjoint} trilinear dual form of the Kakeya maximal operator conjecture

\end{enumerate}
\end{abstract}
\tableofcontents

\section{Introduction}

The purpose of this paper is to prove the equivalence of the following two
Kakeya conjectures in three dimensions, which are described in detail below:

\begin{enumerate}
\item the dual Kakeya maximal operator conjecture,

\item the disjoint trilinear dual Kakeya maximal operator conjecture.
\end{enumerate}

The bilinear analogue of this equivalence was obtained in Tao, Vargas and Vega
\cite[Theorem 3.1]{TaVaVe} in all dimensions equal or greater than
two\footnote{It is possible that the bilinear proof in \cite{TaVaVe} can be
adapted to the \emph{disjoint} trilinear setting, but the Whitney condition in
three dimensions complicates matters, and we will not pursue this approach
here, opting instead to adapt the easier part of the Bourgain and Guth
argument from \cite[Section 2]{BoGu}.}. Of course the trilinear equivalence
implies the bilinear equivalence by H\"{o}lder's inequality.

We give the precise statement in Theorem \ref{big} below, after recalling and
introducing some standard definitions.

\begin{definition}
For $0<\delta<1$, a $\delta$-tube $T$ is a tube with length $1$ and
cross-sectional diameter $\delta$. The \emph{orientation} of a tube $T$ is the
unit vector $\mathbf{v}\left(  T\right)  $ in the direction of the centerline
of $T$ (up to $\pm1$).
\end{definition}

\begin{definition}
A family $\mathbb{T}$ of $\delta$-tubes is $\delta$-separated if $\left\vert
\mathbf{v}\left(  T\right)  -\mathbf{v}\left(  T^{\prime}\right)  \right\vert
\geq\delta$ for all $T,T^{\prime}\in\mathbb{T}$.
\end{definition}

Here is the main dual form of the Kakeya maximal operator inequality we are
interested in, which we denote by $\mathcal{K}^{\ast}\left(  \otimes
_{1}L^{\infty}\rightarrow L^{\frac{3}{2}};\varepsilon\right)  $.

\begin{definition}
Let $0<\varepsilon<1$. We say the statement $\mathcal{K}^{\ast}\left(
\otimes_{1}L^{\infty}\rightarrow L^{\frac{3}{2}};\varepsilon\right)  $ holds
if there is a positive constant $C_{\varepsilon}$ such that%
\begin{align}
&  \left\Vert \sum_{T\in\mathbb{T}}\mathbf{1}_{T}\right\Vert _{L^{\frac{3}{2}%
}\left(  \mathbb{R}^{3}\right)  }\leq C_{\varepsilon}\delta^{-\varepsilon
},\label{Kak dual}\\
&  \text{for all families }\mathbb{T}\text{ of }\delta\text{-separated }%
\delta\text{-tubes in }\mathbb{R}^{3}\text{ and }0<\delta<1.\nonumber
\end{align}

\end{definition}

Here is the main conjecture regarding the Kakeya inequality $\mathcal{K}%
^{\ast}\left(  \otimes_{1}L^{\infty}\rightarrow L^{\frac{3}{2}};\varepsilon
\right)  $. The Kakeya \emph{set} conjecture in $\mathbb{R}^{3}$ (logically
weaker than the Kakeya maximal operator conjecture) has been recently proved
by Hong Wang and Joshua Zahl \cite{WaZa} after being open for decades, and
represents a significant step forward in the understanding of Kakeya theory.

\begin{conjecture}
[Dual form of the Kakeya maximal operator conjecture in $\mathbb{R}^{3}$%
]\label{linear Kakeya}The statement $\mathcal{K}^{\ast}\left(  \otimes
_{1}L^{\infty}\rightarrow L^{\frac{3}{2}};\varepsilon\right)  $ holds for all
$0<\varepsilon<1$.
\end{conjecture}

Now we turn to a \emph{trilinear} variant of the Kakeya conjecture.

\begin{definition}
\label{tube disjoint}Let $0<\delta\leq\nu<1$. Three families $\mathbb{T}_{1}$,
$\mathbb{T}_{2}$ and $\mathbb{T}_{3}$ of $\delta$-tubes are said to be $\nu
$\emph{-disjoint} if the orientations of the $\delta$-tubes $T_{k}%
\in\mathbb{T}_{k}$ are contained in spherical patches $\Omega_{k}%
\subset\mathbb{S}^{2}$, that satisfy%
\begin{equation}
\operatorname*{diam}\left(  \Omega_{k}\right)  \approx\nu\text{ and
}\operatorname*{dist}\left(  \Omega_{j},\Omega_{k}\right)  \geq\nu
,\ \ \ \ \ 1\leq j,k\leq3. \label{nu disjoint}%
\end{equation}
We also say that the triple $\left(  \Omega_{1},\Omega_{2},\Omega_{3}\right)
$ is $\nu$\emph{-disjoint} if (\ref{nu disjoint}) holds.
\end{definition}

\begin{definition}
Let $0<\varepsilon,\nu<1$. We say the statement $\mathcal{K}%
_{\operatorname*{disj}\nu}^{\ast}\left(  \otimes_{3}L^{\infty}\rightarrow
L^{\frac{1}{2}};\varepsilon\right)  $ holds if there is a positive constant
$C_{\varepsilon,\nu}$ such that%
\begin{align}
&  \left\Vert \prod_{k=1}^{3}\left(  \sum_{T_{k}\in\mathbb{T}_{k}}%
\mathbf{1}_{T_{k}}\right)  \right\Vert _{L^{\frac{1}{2}}\left(  \mathbb{R}%
^{3}\right)  }\leq C_{\varepsilon,\nu}\delta^{-\varepsilon}%
,\label{tri Kak dual}\\
&  \text{for all }\nu\text{\emph{-disjoint }families }\mathbb{T}_{k}\text{ of
}\delta\text{-separated }\delta\text{-tubes in }\mathbb{R}^{3}\text{ and
}0<\delta\leq\nu.\nonumber
\end{align}

\end{definition}

\begin{conjecture}
[The disjoint trilinear dual form of the Kakeya maximal operator conjecture in
$\mathbb{R}^{3}$]\label{trilinear Kakeya}The statement $\mathcal{K}%
_{\operatorname*{disj}\nu}^{\ast}\left(  \otimes_{3}L^{\infty}\rightarrow
L^{\frac{1}{2}};\varepsilon\right)  $ holds for all $0<\varepsilon,\nu<1$.
\end{conjecture}

See the references below, and their references as well, for background on
Kakeya conjectures, which were pioneered in large part by work of Tom Wolff
and Jean Bourgain. In particular, we point to the later work on
bilinear\ Kakeya equivalences in\ Tao, Vargas and Vega \cite{TaVaVe}, and the
transverse trilinear theorems proved in Bennett, Carbery and Tao \cite{BeCaTa}.

\begin{theorem}
\label{big}The linear Kakeya Conjecture \ref{linear Kakeya} is equivalent to
the \emph{disjoint} trilinear Kakeya Conjecture \ref{trilinear Kakeya}.
Moreover, these Kakeya conjectures are equivalent to their Fourier
counterparts, Conjectures \ref{ssFsfec} and \ref{ssFsftec} below.
\end{theorem}

To put this theorem into context, we note that if we both \emph{restrict} the
triples $\left(  \mathbb{T}_{1},\mathbb{T}_{2},\mathbb{T}_{3}\right)  $ to be
$\nu$-transverse as in \cite{BeCaTa}, and \emph{permit} parallel and repeated
tubes within the families $\mathbb{T}_{k}$, then it is a \textbf{theorem} of
Bennett, Carbery and Tao \cite[Theorem 1.15]{BeCaTa} that $\left\Vert
\prod_{k=1}^{3}\left(  \sum_{T_{k}\in\mathbb{T}_{k}}\mathbf{1}_{T_{k}}\right)
\right\Vert _{L^{\frac{1}{2}}\left(  \mathbb{R}^{3}\right)  }\leq
C_{\varepsilon,\nu}\delta^{-\varepsilon}$ holds for all $\varepsilon>0$ and
all $\nu$\emph{-transverse }families $\mathbb{T}_{k}$ of $\delta$-tubes in
$\mathbb{R}^{3}$ (not necessarily $\delta$-separated) and $0<\delta\leq\nu$.
Their theorem is proved using a beautiful sliding Gaussian argument, that
depends heavily on the orientations of the tubes being clustered near the
coordinate vectors $\mathbf{e}_{1}$, $\mathbf{e}_{2}$ and $\mathbf{e}_{3}$.
The general transverse case with parallel and repeated tubes is then obtained
by a linear change of variable, but the general \emph{disjoint} case with
separated tubes cannot be obtained in this way\footnote{They then go on to use
a trilinear variant of a bootstrapping argument of Bourgain to prove a
corresponding transverse trilinear Fourier extension theorem \cite[Theorem
1.17]{BeCaTa}, but this transverse result is not yet known to be sufficient
for the Fourier extension conjecture. On the other hand, the logically
stronger \emph{disjoint} trilinear Fourier extension conjecture is shown to be
equivalent to the Fourier extension conjecture in \cite{RiSa}.}.

\subsection{Organization of the proof}

The linear conjecture implies the trilinear conjecture since $\mathcal{K}%
^{\ast}\left(  \otimes_{1}L^{\infty}\rightarrow L^{\frac{3}{2}};\varepsilon
\right)  $ implies $\mathcal{K}_{\operatorname*{disj}\nu}^{\ast}\left(
\otimes_{3}L^{\infty}\rightarrow L^{\frac{1}{2}};\varepsilon\right)  $ for all
$\varepsilon,\nu>0$ by H\"{o}lder's inequality with exponents $\left(
3,3,3\right)  $ (even without the assumption of $\nu$-disjoint in
(\ref{tri Kak dual})). The proof of the converse will proceed by proving the
following chain of equivalences\footnote{strictly speaking we need only prove
the implications in the direction from bottom left to top left as pictured,
but we include the reverse implications since they have informative proofs.},%
\begin{equation}
\fbox{$%
\begin{array}
[c]{ccc}%
\begin{array}
[c]{c}%
\mathbf{Lin\ dual\ Kakeya}\\
\mathbf{Max\ Op\ Conj\ }%
\end{array}
& \Longleftrightarrow &
\begin{array}
[c]{c}%
\mathbf{Lin\ modulated\ single\ scale}\\
\mathbf{Fourier\ Squ\ Ext\ Conj}%
\end{array}
\\
&  & \Updownarrow\\%
\begin{array}
[c]{c}%
\mathbf{Disjoint\ tri\ dual\ Kakeya}\\
\mathbf{Max\ Op\ Conj}%
\end{array}
& \Longleftrightarrow &
\begin{array}
[c]{c}%
\mathbf{Disjoint\ tri\ modulated\ single\ scale}\\
\mathbf{Fourier\ Squ\ Ext\ Conj}%
\end{array}
\end{array}
$}\ . \label{diag}%
\end{equation}

The vertical implication $\Updownarrow$ on the Fourier side of the diagram
(\ref{diag}), comprises the deepest part of the proof, where extension theory
on the paraboloid can be exploited. Indeed, the proof is a modification of the
argument in \cite[proof of Theorem 3]{RiSa}, that followed in turn the
pigeonholing argument of Bourgain and Guth in \cite[Section 2]{BoGu}, and used
the parabolic rescaling in Tao, Vargas and Vega \cite{TaVaVe}. It is this
latter phenomenon that dictates our use of the paraboloid on the Fourier side.
Our variant of this theorem will use the assumption of $\nu$-disjoint patches
as in \cite{RiSa}, called \emph{weak} $\nu$-transverse by Muscalu and Oliveira
in \cite{MuOl}, but will only require the functions $f_{k}$ to have a special
`single scale Kakeya' form $\mathsf{Q}_{s}f\equiv\sum_{I}\left\langle
f,\varphi_{I}\right\rangle \varphi_{I}$, where the functions $\varphi_{I}$ are
$L^{2}$-normalized translations and dilations of a `father' wavelet $\varphi$
to a dyadic square $I$, and where the sum is taken over a tiling by dyadic
squares of side length $2^{-s}$, $s\in\mathbb{N}$. On the other hand, we don't
need to use the induction on scales idea ($K_{1}\ll K$), and the $L^{4}$
estimates going back to Cordoba, that were used in \textbf{Case 3} of the
argument in \cite[Section 2]{BoGu}, since the $\nu$-disjoint assumption we
make in the trilinear inequality includes this difficult case.

While the Kakeya conjectures are formulated at their critical indices, the two
Fourier conjectures require avoiding their critical indices in order to apply
the modification of the Bourgain Guth argument (which requires $q>3$).

The remaining two horizontal equivalences $\Longleftrightarrow$ and
$\Longleftrightarrow$, that each interchange the Kakeya side of the diagram
with the Fourier side, use standard square function and rapid decay arguments.

\subsection{Preliminaries}

The Fourier extension operator $\mathcal{E}$ on the paraboloid in three
dimensions is given by,%
\[
\mathcal{E}f\left(  \xi\right)  \equiv\left[  \Phi_{\ast}\left(  f\left(
x\right)  dx\right)  \right]  ^{\wedge}\left(  \xi\right)  =\int_{U}%
e^{-i\Phi\left(  x\right)  \cdot\xi}f\left(  x\right)  dx,\ \ \ \ \ \text{for
}\xi\in\mathbb{R}^{3},
\]
where $\Phi_{\ast}\left(  f\left(  x\right)  dx\right)  $ denotes the
pushforward of the measure $f\left(  x\right)  dx$ supported in $U\subset
B_{\mathbb{R}^{2}}\left(  0,\frac{1}{2}\right)  $ to the paraboloid
$\mathbb{P}^{2}$ under the usual parameterization $\Phi:U\rightarrow
\mathbb{P}^{2}$ by $\Phi\left(  x\right)  =\left(  x_{1},x_{2},x_{1}^{2}%
+x_{2}^{2}\right)  $ for $x=\left(  x_{1},x_{2}\right)  \in U$, where the
pushforward $\Phi_{\ast}\mu$ of a measure $\mu$ is defined by the identity
$\left\langle g,\Phi_{\ast}\mu\right\rangle =\left\langle g\circ\Phi
,\mu\right\rangle $ for all $g\in C_{c}\left(  \mathbb{R}^{3}\right)  $. We
will often abuse notation and simply write $\Phi_{\ast}f$ where we view $f$ as
the measure $f\left(  x\right)  dx$.

\begin{definition}
Denote by $\operatorname*{Grid}$ the set of all grids $\mathcal{G}$ in
$\mathbb{R}^{2}$, and for $s\in\mathbb{N}$ define%
\[
\mathcal{G}_{s}\left[  U\right]  \equiv\left\{  I\in\mathcal{G}:I\subset
U\text{ and }\ell\left(  I\right)  =2^{-s}\right\}  .
\]

\end{definition}

\begin{notation}
For each $s\in\mathbb{N}$, we fix a $2^{-s}$-separated subset $\mathcal{G}%
_{s}^{\ast}\left[  U\right]  $ of $\mathcal{G}_{s}\left[  U\right]  $. As this
sequence of squares remains fixed throughout the paper - up until the very end
- we will simply write $\mathcal{G}_{s}\left[  U\right]  $, with the
understanding that the squares $I\in\mathcal{G}_{s}\left[  U\right]  $ are
$2^{-s}$-separated. In any event, $\mathcal{G}_{s}\left[  U\right]  $ is a
finite union of collections of the form $\mathcal{G}_{s}^{\ast}\left[
U\right]  $.
\end{notation}

\begin{definition}
\label{def M}Set $\mathcal{V}_{s}\equiv\left\{  \mathbb{R}^{3}\text{-valued
sequences on }\mathcal{G}_{s}\left[  U\right]  \right\}  $. For $s\in
\mathbb{N} $ and $\mathbf{u}=\left\{  u_{I}\right\}  _{I\in\mathcal{G}%
_{s}\left[  U\right]  }\in\mathcal{V}_{s}$, define the modulation
$\mathsf{M}_{\mathbf{u}}^{s}$ on $\Phi\left(  U\right)  $ by,%
\[
\mathsf{M}_{\mathbf{u}}^{s}\left(  z\right)  \equiv\sum_{I\in\mathcal{G}%
_{s}\left[  U\right]  }e^{iu_{I}\cdot z}\mathbf{1}_{\Phi\left(  2I\right)
}\left(  z\right)  ,\ \ \ \ \ \text{\ for }z\in\Phi\left(  U\right)  .
\]

\end{definition}

\begin{definition}
\label{def Q}Let $I_{0}\equiv\left[  0,1\right]  ^{2}$ be the unit square in
the plane. Fix $\varphi\in C_{c}^{\infty}\left(  2I_{0}\right)  $ such that
$\varphi=1$ on $I_{0}$. Then for any square $I$, let $\varphi_{I}$ be the
$L^{2}$ normalized translation and dilation of $\varphi$ that is adapted to
$I$. Given $f\in L^{1}\left(  U\right)  $ and $s\in\mathbb{N}$ define the
pseudoprojections
\[
\bigtriangleup_{I}f\equiv\left\langle f,\varphi_{I}\right\rangle \varphi
_{I}\text{ and }\mathsf{Q}_{s,U}f\equiv\sum_{I\in\mathcal{G}_{s}\left[
U\right]  }\bigtriangleup_{I}f=\sum_{I\in\mathcal{G}_{s}\left[  U\right]
}\left\langle f,\varphi_{I}\right\rangle \varphi_{I}\ .
\]

\end{definition}

Note that the assumption on separation of squares in $\mathcal{G}_{s}\left[
U\right]  =\mathcal{G}_{s}^{\ast}\left[  U\right]  $ implies that $\left\{
\bigtriangleup_{I}f\right\}  _{I\in\mathcal{G}_{s}\left[  U\right]  }$ is a
collection of orthogonal pseudoprojections at level $s$, by which we mean that
for $I,L\in\mathcal{G}_{s}\left[  U\right]  =\mathcal{G}_{s}^{\ast}\left[
U\right]  $,%
\begin{equation}
\bigtriangleup_{L}\bigtriangleup_{I}f=\left\{
\begin{array}
[c]{ccc}%
c_{\flat}\bigtriangleup_{I}f & \text{ if } & L=I\\
0 & \text{ if } & L\not =I
\end{array}
\right.  ,\ \ \ \ \ \text{where }c_{\flat}=\left\langle \varphi_{I}%
,\varphi_{I}\right\rangle =\left\langle \varphi,\varphi\right\rangle \approx1.
\label{pseudo}%
\end{equation}

\section{The Fourier square function}

\begin{definition}
For $s\in\mathbb{N}$ and $\mathbf{u}\in\mathcal{V}_{s}$ define the
\emph{modulated\ single scale }Fourier square function by
\begin{equation}
\mathcal{S}_{\operatorname*{Fourier}}^{s,\mathbf{u}}f\left(  \xi\right)
\equiv\left(  \sum_{I\in\mathcal{G}_{s}\left[  U\right]  }\left\vert \left(
\mathsf{M}_{\mathbf{u}}^{s}\Phi_{\ast}\bigtriangleup_{I}f\right)  ^{\wedge
}\left(  \xi\right)  \right\vert ^{2}\right)  ^{\frac{1}{2}}=\left(
\sum_{I\in\mathcal{G}_{s}\left[  U\right]  }\left\vert \tau_{u_{I}}%
\widehat{\Phi_{\ast}\bigtriangleup_{I}f}\left(  \xi\right)  \right\vert
^{2}\right)  ^{\frac{1}{2}}, \label{def Four square}%
\end{equation}
where $\mathsf{M}_{\mathbf{u}}^{s}$ is the modulation defined in Definition
\ref{def M}, and $\tau_{u_{I}}$ denotes translation by $u_{I}$.
\end{definition}

\begin{definition}
Let $1\leq q\leq\infty$ and $0<\varepsilon<1$. We say the statement
$\mathcal{E}^{\operatorname*{square}}\left(  \otimes_{1}L^{\infty}\rightarrow
L^{q};\varepsilon\right)  $ holds if there is a positive constant
$C_{q,\varepsilon}$ such that%
\begin{equation}
\left\Vert \mathcal{S}_{\operatorname*{Fourier}}^{s,\mathbf{u}}f\right\Vert
_{L^{q}\left(  \mathbb{R}^{3}\right)  }\lesssim C_{q,\varepsilon
}2^{\varepsilon s}\left\Vert f\right\Vert _{L^{\infty}\left(  U\right)
},\ \ \ \ \ \text{for all }s\in\mathbb{N}\text{, }\mathbf{u}=\left\{
u_{I}\right\}  _{I\in\mathcal{G}_{s}\left[  U\right]  }\in\mathcal{V}%
_{s}\text{ and }f\in L^{\infty}\left(  U\right)  . \label{FECUS}%
\end{equation}

\end{definition}

\begin{conjecture}
[modulated single scale Fourier square function extension conjecture]%
\label{ssFsfec}The statement $\mathcal{E}^{\operatorname*{square}}\left(
\otimes_{1}L^{\infty}\rightarrow L^{q};\varepsilon\right)  $ holds for all
$0<\varepsilon<1$ and $q>3$.
\end{conjecture}

Note the use of the sup norm on the right hand side of (\ref{FECUS}), and that
integration is taken over all of $\mathbb{R}^{3}$ on the left hand side, as
opposed to the local inequalities used in \cite{BeCaTa} and \cite{BoGu} that
integrate over balls $B\left(  0,R\right)  $. The role of the radius $R$ is
here played by the scale $s\in\mathbb{N}$, with the connection being that
$\mathbf{1}_{B\left(  0,R\right)  }\leq\varphi_{R}$ for an appropriate bump
function $\varphi$, and $\varphi_{R}\widehat{F}=\widehat{\widehat{\varphi
}_{\frac{1}{R}}\ast F}$ and $\widehat{\varphi}_{\frac{1}{R}}\ast F\sim
\sum_{I\in\mathcal{G}_{s}\left[  U\right]  }\Phi_{\ast}\bigtriangleup_{I}f$
where $F=\sum_{I\in\mathcal{G}_{s}\left[  U\right]  }\bigtriangleup_{I}f$ with
$R=2^{s}$. This will have a consequence for redefining the norm used in
\cite[Section 2.1]{BoGu}. The importance of Conjecture \ref{ssFsfec} lies in
the fact that it implies the dual form of the Kakeya maximal operator
Conjecture (\ref{linear Kakeya}), as we now show.

\begin{definition}
\label{I hat}For $I\in\mathcal{G}_{s}\left[  U\right]  $, denote its center by
$c_{I}$, and denote by $\widehat{I}$ the $\left(  2^{s}\times2^{s}\times
2^{2s}\right)  $-tube in $\mathbb{R}^{3}$ (i.e. having dimensions $2^{s}%
\times2^{s}\times2^{2s}$) centered at the origin with centerline parallel to
the unit upward normal $\mathbf{n}_{\Phi}\left(  \Phi\left(  c_{I}\right)
\right)  $ to the paraboloid at the point $\Phi\left(  c_{I}\right)  $.
\end{definition}

\begin{proposition}
\label{Kak Four}The modulated single scale Fourier square function extension
Conjecture \ref{ssFsfec} implies the dual form of the Kakeya maximal operator
Conjecture \ref{linear Kakeya}.
\end{proposition}

Proofs are well known, but we give one here for the sake of future reference.

\begin{definition}
\label{def tube square}Given a family $\mathbb{T}$ of tubes, the associated
\emph{dual Kakeya square function} $\mathcal{S}\mathbb{T}$ is defined by%
\[
\mathcal{S}\mathbb{T}\left(  \xi\right)  \equiv\left(  \sum_{T\in\mathbb{T}%
}\mathbf{1}_{T}\left(  \xi\right)  ^{2}\right)  ^{\frac{1}{2}}=\left(
\sum_{T\in\mathbb{T}}\mathbf{1}_{T}\left(  \xi\right)  \right)  ^{\frac{1}{2}%
},\ \ \ \ \ \xi\in\mathbb{R}^{3}.
\]

\end{definition}

\begin{proof}
Assume the single scale modular square function inequality (\ref{FECUS}),%
\[
\left\Vert \left(  \sum_{I\in\mathcal{G}_{s}\left[  U\right]  }\left\vert
\tau_{u_{I}}\widehat{\Phi_{\ast}\bigtriangleup_{I}f}\right\vert ^{2}\right)
^{\frac{1}{2}}\right\Vert _{L^{q}\left(  \mathbb{R}^{3}\right)  }=\left\Vert
\left(  \sum_{I\in\mathcal{G}_{s}\left[  U\right]  }\left(  \mathcal{S}%
_{\operatorname*{Fourier}}^{s,\mathbf{u}}f\right)  ^{2}\right)  ^{\frac{1}{2}%
}\right\Vert _{L^{q}\left(  \mathbb{R}^{3}\right)  }\lesssim2^{\varepsilon
s}\left\Vert f\right\Vert _{L^{\infty}\left(  U\right)  },
\]
for all $q>3$ and $0<\varepsilon<1$. The Fourier square function
$\mathcal{S}_{\operatorname*{Fourier}}^{s,\mathbf{u}}f$ dominates $2^{-2s}$
times the corresponding dual Kakeya square function
\[
\mathcal{S}\mathbb{T}\left(  2^{s}\times2^{s}\times2^{2s}\right)  \left(
\xi\right)  \equiv\left(  \sum_{T\in\mathbb{T}\left(  2^{s}\times2^{s}%
\times2^{2s}\right)  }\mathbf{1}_{T}\left(  \xi\right)  \right)  ^{\frac{1}%
{2}},
\]
where the tubes $T$ in $\mathbb{T}$ are the translated tubes $\widehat
{I}+u_{I}$ with $\mathbf{u}=\left\{  u_{I}\right\}  _{I\in\mathcal{G}%
_{s}\left[  U\right]  }$. Here we have written $\mathbb{T}=\mathbb{T}\left(
2^{s}\times2^{s}\times2^{2s}\right)  $ to denote a family of $2^{-s}%
$-separated $\left(  2^{s}\times2^{s}\times2^{2s}\right)  $-tubes, where an
$\left(  \alpha\times\beta\times\gamma\right)  $-tube has dimensions
$\alpha\times\beta\times\gamma$ and is oriented in any direction. Indeed, we
have%
\begin{align*}
\left\vert \widehat{\Phi_{\ast}\bigtriangleup_{I}f}\left(  \xi\right)
\right\vert  &  \gtrsim\left\Vert \bigtriangleup_{I}f\right\Vert _{L^{1}%
}\mathbf{1}_{\widehat{I}}\left(  \xi\right)  \approx\left\vert \left\langle
f,\varphi_{I}\right\rangle \right\vert \left\Vert \varphi_{I}\right\Vert
_{L^{1}}\mathbf{1}_{\widehat{I}}\left(  \xi\right) \\
&  \gtrsim\left\Vert f\right\Vert _{L^{\infty}\left(  U\right)  }\left\Vert
\varphi_{I}\right\Vert _{L^{1}}\left\Vert \varphi_{I}\right\Vert _{L^{1}%
}\mathbf{1}_{\widehat{I}}\left(  \xi\right)  \approx2^{-2s}\mathbf{1}%
_{\widehat{I}}\left(  \xi\right)  ,
\end{align*}
for an appropriate choice of $f$ with $\left\Vert f\right\Vert _{L^{\infty
}\left(  U\right)  }=1$. Altogether, there is $0<C<\infty$ such that%
\begin{equation}
\mathcal{S}\mathbb{T}\left(  2^{s}\times2^{s}\times2^{2s}\right)  \left(
\xi\right)  \leq C2^{2s}\mathcal{S}_{\operatorname*{Fourier}}^{s,\mathbf{u}%
}f\left(  \xi\right)  ,\ \ \ \ \ \text{for all }\xi\in\mathbb{R}^{3},
\label{square point}%
\end{equation}
where $f$ is an appropriate function bounded by $1$, and supported on a set of
measure $2^{-2s}\#\mathbb{T}\lesssim1$.

Thus we have%
\begin{align*}
&  \left(  \int_{\mathbb{R}^{3}}\left(  \sum_{T\in\mathbb{T}\left(
2^{s}\times2^{s}\times2^{2s}\right)  }\mathbf{1}_{T}\left(  \xi\right)
\right)  ^{p}d\xi\right)  ^{\frac{1}{p}}=\left(  \int_{\mathbb{R}^{3}%
}\mathcal{S}\mathbb{T}\left(  2^{s}\times2^{s}\times2^{2s}\right)  \left(
\xi\right)  ^{2p}d\xi\right)  ^{\frac{1}{p}}\\
&  \lesssim\left(  \int_{\mathbb{R}^{3}}\left[  2^{2s}\mathcal{S}%
_{\operatorname*{Fourier}}^{s,\mathbf{u}}f\left(  \xi\right)  \right]
^{2p}d\xi\right)  ^{\frac{1}{p}}=2^{4s}\left(  \int_{\mathbb{R}^{3}}\left[
\mathcal{S}_{\operatorname*{Fourier}}^{s,\mathbf{u}}f\left(  \xi\right)
\right]  ^{2p}d\xi\right)  ^{\frac{2}{2p}}\\
&  =\left(  2^{2s}\left\Vert \mathcal{S}_{\operatorname*{Fourier}%
}^{s,\mathbf{u}}f\right\Vert _{L^{2p}}\right)  ^{2}\lesssim\left(
2^{2s}\left\Vert f\right\Vert _{L^{\infty}\left(  U\right)  }\right)  ^{2},
\end{align*}
by Conjecture \ref{ssFsfec}\ for $q=2p>3$, i.e.%
\[
\left(  \int_{\mathbb{R}^{3}}\left(  \sum_{T\in\mathbb{T}\left(  2^{s}%
\times2^{s}\times2^{2s}\right)  }\mathbf{1}_{T}\left(  \xi\right)  \right)
^{p}d\xi\right)  ^{\frac{1}{2p}}\lesssim2^{2s}\left\Vert f\right\Vert
_{L^{\infty}\left(  U\right)  }\leq2^{2s}.
\]
Scaling this down by $2^{-2s}$ to $\left(  2^{-s}\times2^{-s}\times1\right)
$-tubes gives%
\begin{align*}
\left(  \int_{\mathbb{R}^{3}}\left(  \sum_{T\in\mathbb{T}\left(  2^{-s}%
\times2^{-s}\times1\right)  }\mathbf{1}_{T}\left(  \xi\right)  \right)
^{p}\left(  2^{2s}\right)  ^{3}d\xi\right)  ^{\frac{1}{2p}}  &  \lesssim
2^{2s},\\
\text{i.e. }\left(  \int_{\mathbb{R}^{3}}\left(  \sum_{T\in\mathbb{T}\left(
2^{-s}\times2^{-s}\times1\right)  }\mathbf{1}_{T}\left(  \xi\right)  \right)
^{p}d\xi\right)  ^{\frac{1}{p}}  &  \lesssim2^{\left(  2-\frac{3}{p}\right)
s},\ \ \ \ \ \text{for }p>\frac{3}{2}.
\end{align*}

Now `interpolate' with respect to the trivial $L^{1}$ inequality $\left\Vert
\sum_{T\in\mathbb{T}\left(  2^{-s}\times2^{-s}\times1\right)  }\mathbf{1}%
_{T}\right\Vert _{L^{1}}\leq1$, to obtain (\ref{Kak dual}) with $\delta
=2^{-s}$. Indeed, given $\varepsilon>0$, let $p_{\varepsilon}=\frac{3}%
{2}+\varepsilon>\frac{3}{2}$, so that,%
\begin{align}
&  \int_{\mathbb{R}^{3}}\left(  \sum_{T\in\mathbb{T}\left(  2^{-s}\times
2^{-s}\times1\right)  }\mathbf{1}_{T}\left(  \xi\right)  \right)  ^{\frac
{3}{2}}d\xi=\left\{  \int_{\left\{  \sum_{T\in\mathbb{T}}\mathbf{1}_{T}%
\leq1\right\}  }+\int_{\left\{  \sum_{T\in\mathbb{T}}\mathbf{1}_{T}>1\right\}
}\right\}  \left(  \sum_{T\in\mathbb{T}\left(  2^{-s}\times2^{-s}%
\times1\right)  }\mathbf{1}_{T}\left(  \xi\right)  \right)  ^{\frac{3}{2}}%
d\xi\label{interp}\\
&  \leq\int_{\left\{  \sum_{T\in\mathbb{T}}\mathbf{1}_{T}\leq1\right\}
}\left(  \sum_{T\in\mathbb{T}\left(  2^{-s}\times2^{-s}\times1\right)
}\mathbf{1}_{T}\left(  \xi\right)  \right)  d\xi+\int_{\left\{  \sum
_{T\in\mathbb{T}}\mathbf{1}_{T}>1\right\}  }\left(  \sum_{T\in\mathbb{T}%
\left(  2^{-s}\times2^{-s}\times1\right)  }\mathbf{1}_{T}\left(  \xi\right)
\right)  ^{p_{\varepsilon}}d\xi\nonumber\\
&  \lesssim1+C_{p_{\varepsilon}}2^{\left(  4p_{\varepsilon}-6\right)
s}=1+C_{p_{\varepsilon}}2^{4\varepsilon s}.\nonumber
\end{align}
This completes the proof that Conjecture \ref{ssFsfec} implies Conjecture
\ref{linear Kakeya}.
\end{proof}

Here is the definition of a $\nu$-disjoint triple of squares that is analogous
to that of families of tubes given in Definition \ref{tube disjoint}.

\begin{definition}
A triple $\left(  U_{1},U_{2},U_{3}\right)  $ of squares $U_{k}\subset U$ is
$\nu$\emph{-disjoint} if%
\begin{equation}
\ell\left(  U_{k}\right)  \approx\nu,\text{ and }\operatorname*{dist}\left(
U_{j},U_{k}\right)  \geq\nu,\ \ \ \ \ 1\leq j,k\leq3. \label{nu disjoint'}%
\end{equation}

\end{definition}

\begin{definition}
Let $1<q<\infty$ and $0<\varepsilon,\nu<1$. We say the statement
$\mathcal{A}_{\operatorname*{disj}\nu}^{\operatorname*{square}}\left(
\otimes_{3}L^{\infty}\rightarrow L^{\frac{q}{3}};\varepsilon\right)  $ holds
if there is a positive constant $C_{q,\varepsilon,\nu}$ depending only on $q$,
$\varepsilon$ and $\nu$, such that,%
\begin{align}
&  \left\Vert \mathcal{S}_{\operatorname*{Fourier}}^{s,\mathbf{u}_{1}}%
f_{1}\ \mathcal{S}_{\operatorname*{Fourier}}^{s,\mathbf{u}_{2}}f_{2}%
\ \mathcal{S}_{\operatorname*{Fourier}}^{s,\mathbf{u}_{3}}f_{3}\right\Vert
_{L^{\frac{q}{3}}\left(  \mathbb{R}^{3}\right)  }\leq C_{q,\varepsilon,\nu
}2^{\varepsilon s}\left\Vert f_{1}\right\Vert _{L^{\infty}\left(  U\right)
}\left\Vert f_{2}\right\Vert _{L^{\infty}\left(  U\right)  }\left\Vert
f_{3}\right\Vert _{L^{\infty}\left(  U\right)  }\ ,\label{single tri Four}\\
\text{for all }s  &  \in\mathbb{N}\text{ with }2^{-s}\leq\nu\text{, all }%
f_{k}\in L^{\infty}\left(  U_{k}\right)  \text{, all sequences }\mathbf{u}%
_{k}\in\mathcal{V}\text{, and all }\nu\text{-disjoint triples }\left(
U_{1},U_{2},U_{3}\right)  \subset U^{3}.\nonumber
\end{align}

\end{definition}

\begin{conjecture}
[modulated single scale Fourier square function disjoint trilinear extension
conjecture]\label{ssFsftec}For every $q>3$ there is $0<\nu<1$ such that the
statement $\mathcal{A}_{\operatorname*{disj}\nu}^{\operatorname*{square}%
}\left(  \otimes_{3}L^{\infty}\rightarrow L^{\frac{q}{3}};\varepsilon\right)
$ holds for all $0<\varepsilon<1$.
\end{conjecture}

\begin{theorem}
[Fourier square function analogue of Theorem \ref{big}]\label{SFA}The
modulated single scale Fourier square function extension Conjecture
\ref{ssFsfec} for the paraboloid $\mathbb{P}^{2}$ in $\mathbb{R}^{3}$ holds
\emph{if and only if} the modulated single scale Fourier square function
disjoint trilinear extension Conjecture \ref{ssFsftec} holds.
\end{theorem}

The `only if' assertion follows from H\"{o}lder's inequality with exponents
$\left(  3,3,3\right)  $,
\begin{align*}
&  \left\Vert \mathcal{S}_{\operatorname*{Fourier}}^{s,\mathbf{u}_{1}}%
f_{1}\ \mathcal{S}_{\operatorname*{Fourier}}^{s,\mathbf{u}_{2}}f_{2}%
\ \mathcal{S}_{\operatorname*{Fourier}}^{s,\mathbf{u}_{3}}f_{3}\right\Vert
_{L^{\frac{q}{3}}\left(  \mathbb{R}^{3}\right)  }^{\frac{q}{3}}=\int
_{\mathbb{R}^{3}}\left(  \mathcal{S}_{\operatorname*{Fourier}}^{s,\mathbf{u}%
_{1}}f_{1}\right)  ^{\frac{q}{3}}\left(  \xi\right)  \ \left(  \mathcal{S}%
_{\operatorname*{Fourier}}^{s,\mathbf{u}_{2}}f_{2}\right)  ^{\frac{q}{3}%
}\left(  \xi\right)  \ \left(  \mathcal{S}_{\operatorname*{Fourier}%
}^{s,\mathbf{u}_{3}}f_{3}\right)  ^{\frac{q}{3}}\left(  \xi\right)  d\xi\\
&  \ \ \ \ \ \ \ \ \ \ \ \ \ \ \ \leq\left(  \int_{\mathbb{R}^{3}}\left(
\mathcal{S}_{\operatorname*{Fourier}}^{s,\mathbf{u}_{1}}f_{1}\right)
^{q}\left(  \xi\right)  d\xi\right)  ^{\frac{1}{3}}\ \left(  \int
_{\mathbb{R}^{3}}\left(  \mathcal{S}_{\operatorname*{Fourier}}^{s,\mathbf{u}%
_{2}}f_{2}\right)  ^{q}\left(  \xi\right)  d\xi\right)  ^{\frac{1}{3}%
}\ \left(  \int_{\mathbb{R}^{3}}\left(  \mathcal{S}_{\operatorname*{Fourier}%
}^{s,\mathbf{u}_{3}}f_{3}\right)  ^{q}\left(  \xi\right)  d\xi\right)
^{\frac{1}{3}}\\
&  \ \ \ \ \ \ \ \ \ \ \ \ \ \ \ \leq\left(  C_{\varepsilon,q}2^{\varepsilon
s}\left\Vert f_{1}\right\Vert _{L^{\infty}\left(  U_{1}\right)  }\right)
^{\frac{q}{3}}\ \left(  C_{\varepsilon,q}2^{\varepsilon s}\left\Vert
f_{2}\right\Vert _{L^{\infty}\left(  U_{2}\right)  }\right)  ^{\frac{q}{3}%
}\ \left(  C_{\varepsilon,q}2^{\varepsilon s}\left\Vert f_{3}\right\Vert
_{L^{\infty}\left(  U_{3}\right)  }\right)  ^{\frac{q}{3}},
\end{align*}
for any choice of $U_{k}$, and so in particular, for \emph{any} $0<\nu<1$ we
have,%
\[
\left\Vert \mathcal{S}_{\operatorname*{Fourier}}^{s,\mathbf{u}_{1}}%
f_{1}\ \mathcal{S}_{\operatorname*{Fourier}}^{s,\mathbf{u}_{2}}f_{2}%
\ \mathcal{S}_{\operatorname*{Fourier}}^{s,\mathbf{u}_{3}}f_{3}\right\Vert
_{L^{\frac{q}{3}}\left(  \mathbb{R}^{3}\right)  }\leq C_{\varepsilon,q}%
^{3}2^{3\varepsilon s}\left\Vert f_{1}\right\Vert _{L^{\infty}\left(
U_{1}\right)  }\left\Vert f_{2}\right\Vert _{L^{\infty}\left(  U_{1}\right)
}\left\Vert f_{3}\right\Vert _{L^{\infty}\left(  U_{1}\right)  }.
\]
for all $\nu$-disjoint triples $\left(  U_{1},U_{2},U_{3}\right)  \subset
U^{3}$.

The `if' assertion is a modification of the corresponding proof in \cite[proof
of Theorem 3]{RiSa}, which was in turn based on the pigeonholing argument in
\cite[Section 2]{BoGu}, in which bounded functions are here replaced with
single scale wavelet pseudoprojections.

We give a brief overview of the argument here, and present complete details in
the appendix below. More precisely, for $1<q<\infty$ and $s\in\mathbb{N} $, we
replace the local quantity $Q_{R}^{\left(  q\right)  }$ defined in \cite[proof
of Theorem 3]{RiSa}, and originally in \cite[Section 2]{BoGu}, with the single
scale quantity
\[
A_{s}^{\left(  q\right)  }\equiv\sup_{\mathcal{G}\in\operatorname*{Grid}}%
\sup_{\left\Vert f\right\Vert _{L^{\infty}\left(  U\right)  }\leq1}%
\sup_{\mathbf{u}\in\mathcal{V}}\left(  \int_{\mathbb{R}^{3}}\sum
_{I\in\mathcal{G}_{s}\left[  U\right]  }\left\vert \left(  \mathsf{M}%
_{\mathbf{u}}^{s}\Phi_{\ast}\bigtriangleup_{I}f\right)  ^{\wedge}\left(
\xi\right)  \right\vert ^{q}d\xi\right)  ^{\frac{1}{q}},
\]
where we now take a supremum over the modulations, and over all grids
$\mathcal{G}$, since a parabolic rescaling may change both the modulation and
the grid as we discuss below.

First note that $\left\Vert \bigtriangleup_{I}f\right\Vert _{L^{\infty}%
}\lesssim\left\Vert f\right\Vert _{L^{\infty}}\leq1$, so that we can start the
proof as in \cite[proof of Theorem 3]{RiSa}, which followed \cite[Section
2]{BoGu} almost verbatim at this point, but adapted to square functions and
modulations, and then proceed with \textbf{Case 1} in the same way as well.

The key to the success of the rescaling argument of Tao, Vargas and Vega in
\cite{TaVaVe} applied to $A_{s}^{\left(  q\right)  }$, instead of
$Q_{R}^{\left(  q\right)  }$, that is used in the \textbf{Case 2 and 3}
arguments, is that the projection $\bigtriangleup_{I}$ at a given scale $s$,
when rescaled by a dyadic number $\rho=2^{-n}$, is again a pseudoprojection at
scale $s-n$, but in a possibly different grid. Indeed, this is a consequence
of the translation and dilation invariance of the wavelets $\varphi_{I}$.
Writing $f_{\rho}\left(  x\right)  =f\left(  \rho x\right)  $, we have%
\begin{align*}
\left(  \varphi_{I}\right)  _{\rho}\left(  x\right)   &  =\frac{1}{\rho
}\varphi_{\frac{1}{\rho}I}\left(  x\right)  ,\\
\left\langle f,\varphi_{I}\right\rangle  &  =\int f\left(  \rho x\right)
\varphi_{I}\left(  \rho x\right)  d\rho x=\rho^{2}\left\langle f_{\rho
},\left(  \varphi_{I}\right)  _{\rho}\right\rangle ,
\end{align*}
so that%
\begin{align}
\left(  \bigtriangleup_{I}f\right)  _{\rho}\left(  x\right)   &
=\bigtriangleup_{I}f\left(  \rho x\right)  =\left\langle f,\varphi
_{I}\right\rangle \varphi_{I}\left(  \rho x\right)  =\rho^{2}\left\langle
f_{\rho},\left(  \varphi_{I}\right)  _{\rho}\right\rangle \left(  \varphi
_{I}\right)  _{\rho}\left(  x\right) \label{par res}\\
&  =\left\langle f_{\rho},\rho\left(  \varphi_{I}\right)  _{\rho}\right\rangle
\rho\left(  \varphi_{I}\right)  _{\rho}\left(  x\right)  =\left\langle
f_{\rho},\varphi_{\frac{1}{\rho}I}\right\rangle \varphi_{\frac{1}{\rho}%
I}\left(  x\right)  =\left(  \bigtriangleup_{\frac{1}{\rho}I}\left(  f_{\rho
}\right)  \right)  \left(  x\right)  .\nonumber
\end{align}

As mentioned above, the dilates relative to the origin, of the level $s$
squares of a grid $\mathcal{G}$ by the factor $2^{n}$, are the level $s-n$
squares of a possibly different grid $\mathcal{G}^{\prime}$, so that
\[
\operatorname*{dil}_{2^{n}}\mathcal{G}_{s}\left[  U_{1}\right]  \subset
\mathcal{G}_{s-n}^{\prime}\left[  U\right]  ,\ \ \ \ \ \text{if }2^{n}%
U_{1}\subset U.
\]
It follows that the family of singular measures on the paraboloid that are
pushforwards $\Phi_{\ast}\left(  \sum_{I\in\mathcal{G}_{s}\left[
U_{1}\right]  }\varphi_{I}\left(  x\right)  dx\right)  $ of the planar
measures $\sum_{I\in\mathcal{G}_{s}\left[  U_{1}\right]  }\varphi_{I}\left(
x\right)  dx $, are again of this form under parabolic rescalings $\rho
=2^{-n}$, provided that $2^{n}U_{1}\subset U$. The same is true for parabolic
dilations relative to other points on the paraboloid. Finally, the set of
compositions $\mathsf{M}_{\mathbf{u}}^{s}\Phi_{\ast}$ of modulations
$\mathsf{M}_{\mathbf{u}}^{s}$ with the pushforward $\Phi_{\ast}$, are
preserved under dyadic parabolic rescalings, and so altogether, the quantities
$A_{R}^{\left(  q\right)  }$ scale just as in the argument in \cite[Case 2 of
the proof of Theorem 3]{RiSa}, see also \cite[Section 2]{BoGu}, and
\textbf{Case 2} can proceed as in \cite[proof of Theorem 3]{RiSa}.

The argument in \textbf{Case 3} of \cite[proof of Theorem 3]{RiSa} again
relies on parabolic rescaling, and one can easily complete the proof of
Theorem \ref{SFA} with the above modifications in mind. Complete details of
this argument are given in the final section of the paper.

\section{Linear and trilinear equivalence of Kakeya maximal operator
conjectures}

Here we prove the two horizontal equivalences $\Longleftrightarrow$ in diagram
(\ref{diag}), namely the equivalence of the single scale modulated Fourier
square function extension conjecture with the dual form of the Kakeya maximal
operator conjecture, in both the linear setting and the trilinear setting.

Given a family of tubes $\mathbb{T}$, recall the definition of the tube square
function $\mathcal{S}\mathbb{T}$ in Definition \ref{def tube square}, and
given a square $I\in\mathcal{G}_{s}\left[  U\right]  $, recall the definition
of the tube $\widehat{I}$ in Definition \ref{I hat}.

\begin{proposition}
\label{Kak equiv}The dual form of the Kakeya maximal operator Conjecture
\ref{linear Kakeya} holds \emph{if and only if} the single scale modulated
Fourier square function extension Conjecture \ref{ssFsfec} holds.
\end{proposition}

\begin{proof}
The `if' assertion was already proved in Proposition (\ref{Kak dual}) above.
Conversely, given $\varepsilon>0$, we use the rapid decay of the Fourier
transform away from the associated tubes, to obtain that%
\begin{align}
&  \left\Vert \mathcal{S}_{\operatorname*{Fourier}}^{s,\mathbf{u}}f\right\Vert
_{L^{3}\left(  \mathbb{R}^{3}\right)  }=\left\Vert \left(  \sum_{I\in
\mathcal{G}_{s}\left[  U\right]  }\left\vert \tau_{u_{I}}\widehat{\Phi_{\ast
}\bigtriangleup_{I}f}\left(  \xi\right)  \right\vert ^{2}\right)  ^{\frac
{1}{2}}\right\Vert _{L^{3}\left(  \mathbb{R}^{3}\right)  }\label{rap}\\
&  \lesssim2^{-2s}\left(  \int_{\mathbb{R}^{3}}\left(  \sum_{I\in
\mathcal{G}_{s}\left[  U\right]  }\left\vert \mathbf{1}_{\tau_{u_{I}}%
\widehat{I}}\left(  \xi\right)  \right\vert ^{2}\right)  ^{\frac{3}{2}}%
d\xi\right)  ^{\frac{1}{3}}\left\Vert f\right\Vert _{L^{\infty}\left(
U\right)  }\nonumber\\
&  \ \ \ \ \ \ \ \ \ \ \ \ \ \ \ \ \ \ \ \ \ \ \ \ \ \ \ \ \ \ +\sum
_{m=1}^{\infty}a_{m}2^{-2s}\left(  \int_{\mathbb{R}^{3}}\left(  \sum
_{I\in\mathcal{G}_{s}\left[  U\right]  }\left\vert \mathbf{1}_{2^{m}%
\tau_{u_{I}}\widehat{I}}\left(  \xi\right)  \right\vert ^{2}\right)
^{\frac{3}{2}}d\xi\right)  ^{\frac{1}{3}}\left\Vert f\right\Vert _{L^{\infty
}\left(  U\right)  }\nonumber
\end{align}
which using the change of variable $\xi\rightarrow2^{2s}\xi$, is at most%
\begin{align*}
&  2^{-2s}\left(  \int_{\mathbb{R}^{3}}\left(  \sum_{I\in\mathcal{G}%
_{s}\left[  U\right]  }\left\vert \mathbf{1}_{2^{-2s}\tau_{u_{I}}\widehat{I}%
}\left(  \xi\right)  \right\vert ^{2}\right)  ^{\frac{3}{2}}\left(
2^{2s}\right)  ^{3}d\xi\right)  ^{\frac{1}{3}}\left\Vert f\right\Vert
_{L^{\infty}\left(  U\right)  }\\
&  \ \ \ \ \ \ \ \ \ \ \ \ \ \ \ \ \ \ \ \ \ \ \ \ \ \ \ \ \ \ +\sum
_{m=1}^{\infty}a_{m}2^{-2s}\left(  \int_{\mathbb{R}^{3}}\left(  \sum
_{I\in\mathcal{G}_{s}\left[  U\right]  }\left\vert \mathbf{1}_{2^{m}%
2^{-2s}\tau_{u_{I}}\widehat{I}}\left(  \xi\right)  \right\vert ^{2}\right)
^{\frac{3}{2}}\left(  2^{2s}\right)  ^{3}d\xi\right)  ^{\frac{1}{3}}\left\Vert
f\right\Vert _{L^{\infty}\left(  U\right)  }\\
&  =\left(  \int_{\mathbb{R}^{3}}\left(  \sum_{I\in\mathcal{G}_{s}\left[
U\right]  }\mathbf{1}_{T_{I}}\left(  \xi\right)  \right)  ^{\frac{3}{2}}%
d\xi\right)  ^{\frac{1}{3}}\left\Vert f\right\Vert _{L^{\infty}\left(
U\right)  }+\sum_{m=1}^{\infty}a_{m}\left(  \int_{\mathbb{R}^{3}}\left(
\sum_{I\in\mathcal{G}_{s}\left[  U\right]  }\mathbf{1}_{2^{m}T_{I}}\left(
\xi\right)  \right)  ^{\frac{3}{2}}d\xi\right)  ^{\frac{1}{3}}\left\Vert
f\right\Vert _{L^{\infty}\left(  U\right)  }\\
&  \lesssim\left\Vert \mathcal{S}\mathbb{T}\right\Vert _{L^{\frac{3}{2}%
}\left(  \mathbb{R}^{3}\right)  }^{\frac{1}{2}}\left\Vert f\right\Vert
_{L^{\infty}\left(  U\right)  }+\sum_{m=1}^{\infty}a_{m}C_{0}^{m}\left(
\int_{\mathbb{R}^{3}}\mathcal{S}\mathbb{T}\left(  \xi\right)  ^{\frac{3}{2}%
}d\xi\right)  ^{\frac{1}{3}}\left\Vert f\right\Vert _{L^{\infty}\left(
U\right)  }\\
&  \lesssim\left(  \left\Vert \mathcal{S}\mathbb{T}\right\Vert _{L^{\frac
{3}{2}}\left(  \mathbb{R}^{3}\right)  }^{\frac{1}{2}}+\sum_{m=1}^{\infty}%
a_{m}C_{0}^{m}\left\Vert \mathcal{S}\mathbb{T}\right\Vert _{L^{\frac{3}{2}%
}\left(  \mathbb{R}^{3}\right)  }^{\frac{1}{2}}\right)  \left\Vert
f\right\Vert _{L^{\infty}\left(  U\right)  }\lesssim\sqrt{C_{\varepsilon
}\delta^{-\varepsilon}}\left\Vert f\right\Vert _{L^{\infty}\left(  U\right)
}\lesssim2^{\frac{1}{2}\varepsilon s}\left\Vert f\right\Vert _{L^{\infty
}\left(  U\right)  },
\end{align*}
where $\delta=2^{-s}$, since the coefficients $a_{m}$ are rapidly decreasing.
Here $C_{0}$ is a positive constant related to the geometry of expanded tubes.
Now we `interpolate' this estimate
\[
\left\Vert \mathcal{S}_{\operatorname*{Fourier}}^{s,\mathbf{u}}f\right\Vert
_{L^{3}\left(  \mathbb{R}^{3}\right)  }\lesssim2^{\frac{1}{2}\varepsilon
s}\left\Vert f\right\Vert _{L^{\infty}\left(  U\right)  }\ ,
\]
with the trivial $L^{\infty}$ estimate%
\begin{align*}
&  \left\Vert \mathcal{S}_{\operatorname*{Fourier}}^{s,\mathbf{u}}f\right\Vert
_{L^{\infty}\left(  \mathbb{R}^{3}\right)  }\leq\sup_{\xi\in\mathbb{R}^{3}%
}\mathcal{S}_{\operatorname*{Fourier}}^{s,\mathbf{u}}f\left(  \xi\right)
\leq\left(  \sum_{I\in\mathcal{G}_{s}\left[  U\right]  }\left\vert \tau
_{u_{I}}\widehat{\Phi_{\ast}\bigtriangleup_{I}f}\left(  \xi\right)
\right\vert ^{2}\right)  ^{\frac{1}{2}}\\
&  \leq\left(  \sum_{I\in\mathcal{G}_{s}\left[  U\right]  }\left\Vert
\Phi_{\ast}\bigtriangleup_{I}f\right\Vert ^{2}\right)  ^{\frac{1}{2}}%
\lesssim\left(  \sum_{I\in\mathcal{G}_{s}\left[  U\right]  }\left\Vert
\bigtriangleup_{I}f\right\Vert _{L^{1}\left(  \mathbb{R}^{3}\right)  }%
^{2}\right)  ^{\frac{1}{2}}\lesssim\left(  \sum_{I\in\mathcal{G}_{s}\left[
U\right]  }2^{-2s}\left\Vert f\right\Vert _{L^{\infty}\left(  U\right)  }%
^{2}\right)  ^{\frac{1}{2}}\lesssim\left\Vert f\right\Vert _{L^{\infty}\left(
U\right)  }\ ,
\end{align*}
to obtain (\ref{FECUS}) for all $q>3$, and with a growth factor
$2^{\varepsilon^{\prime}s}$ for $0<\varepsilon^{\prime}<1$ arbitrarily small.
\end{proof}

\begin{proposition}
\label{Kak tri equiv}The disjoint trilinear dual form of the Kakeya maximal
operator Conjecture \ref{trilinear Kakeya} holds \emph{if and only if} the
modulated single scale disjoint\emph{\ }trilinear Fourier square function
Conjecture \ref{ssFsftec} holds.
\end{proposition}

\begin{proof}
The `if' assertion is similar to the argument proving (\ref{Kak dual}) in
Proposition \ref{Kak Four}, but with trilinearity in place of linearity.
Indeed, given $0<\varepsilon,\nu<1$, we must show that $\mathcal{K}%
_{\operatorname*{disj}\nu}^{\ast}\left(  \otimes_{3}L^{\infty}\rightarrow
L^{\frac{3}{2}};\varepsilon\right)  $ holds. Since $\mathcal{K}%
_{\operatorname*{disj}\nu^{\prime}}^{\ast}\left(  \otimes_{3}L^{\infty
}\rightarrow L^{\frac{3}{2}};\varepsilon^{\prime}\right)  $ implies
$\mathcal{K}_{\operatorname*{disj}\nu}^{\ast}\left(  \otimes_{3}L^{\infty
}\rightarrow L^{\frac{3}{2}};\varepsilon\right)  $ for $\varepsilon^{\prime
}\leq\varepsilon$ and $\nu^{\prime}\leq\nu$, it is enough to show that there
exists $\nu^{\prime}\in\left(  0,\nu\right)  $ such that $\mathcal{K}%
_{\operatorname*{disj}\nu^{\prime}}^{\ast}\left(  \otimes_{3}L^{\infty
}\rightarrow L^{\frac{3}{2}};\varepsilon^{\prime}\right)  $ holds for all
$0<\varepsilon^{\prime}<1 $. For this we fix $q_{\varepsilon}\in\left(
3,3+\varepsilon\right)  $. Then by Conjecture \ref{ssFsfec} there is
$\nu^{\prime}>0$ depending on $q_{\varepsilon}$ such that
(\ref{single tri Four}) holds for all $0<\varepsilon^{\prime}<1$, and in light
of Remark \ref{param}, we may assume $\nu^{\prime}<\nu$. Thus if $\left(
\mathbb{T}_{1},\mathbb{T}_{2},\mathbb{T}_{3}\right)  $ is a $\nu^{\prime}%
$-disjoint triple of families of $2^{-s}$-separated $2^{-s}$-tubes, and if
$\left(  U_{1},U_{2},U_{3}\right)  $ is a $\nu^{\prime}$-disjoint triple of
squares in $\mathcal{G}\left[  U\right]  $ such that the orientations of the
tubes in $\mathbb{T}_{k}$ are normal to some point on the surface $\Phi\left(
U_{k}\right)  $, we have%
\begin{align*}
&  \left(  \int_{\mathbb{R}^{3}}\left(  \mathcal{S}_{\operatorname*{Fourier}%
}^{s,\mathbf{u}_{1}}f_{1}\left(  \xi\right)  \ \mathcal{S}%
_{\operatorname*{Fourier}}^{s,\mathbf{u}_{2}}f_{2}\left(  \xi\right)
\ \mathcal{S}_{\operatorname*{Fourier}}^{s,\mathbf{u}_{3}}f_{3}\left(
\xi\right)  \right)  ^{\frac{q_{\varepsilon}}{3}}\ d\xi\right)  ^{\frac
{3}{q_{\varepsilon}}}\\
&  \ \ \ \ \ \ \ \ \ \ \ \ \ \ \ \ \ \ \ \ \lesssim2^{\varepsilon^{\prime}%
s}\left\Vert f_{1}\right\Vert _{L^{\infty}\left(  U_{1}\right)  }\left\Vert
f_{2}\right\Vert _{L^{\infty}\left(  U_{2}\right)  }\left\Vert f_{3}%
\right\Vert _{L^{\infty}\left(  U_{3}\right)  }\ ,
\end{align*}
with $f_{k}\equiv\sum_{I\in G_{s}\left[  U_{k}\right]  }\bigtriangleup_{I}1$.
Then using the pointwise square function estimate (\ref{square point}),
\[
\mathcal{S}\mathbb{T}\left(  2^{s}\times2^{s}\times2^{2s}\right)  \left(
\xi\right)  \leq C2^{2s}\mathcal{S}_{\operatorname*{Fourier}}^{s,\mathbf{u}%
}f\left(  \xi\right)  ,
\]
we obtain from this that,%
\begin{align*}
&  2^{-6s}\left\Vert \prod_{k=1}^{3}\mathcal{S}\mathbb{T}_{k}\left(
2^{s}\times2^{s}\times2^{2s}\right)  \right\Vert _{L^{\frac{q_{\varepsilon}%
}{3}}\left(  \mathbb{R}^{3}\right)  }=\left\Vert \prod_{k=1}^{3}%
2^{-2s}\mathcal{S}\mathbb{T}_{k}\left(  2^{s}\times2^{s}\times2^{2s}\right)
\right\Vert _{L^{\frac{q_{\varepsilon}}{3}}\left(  \mathbb{R}^{3}\right)  }\\
&  \lesssim\left\Vert \prod_{k=1}^{3}\mathcal{S}_{\operatorname*{Fourier}%
}^{s,\mathbf{u}_{k}}f_{k}\left(  \xi\right)  \right\Vert _{L^{\frac{q_{3}}{3}%
}\left(  \mathbb{R}^{3}\right)  }\lesssim2^{\varepsilon^{\prime}s}\left\Vert
f_{1}\right\Vert _{L^{\infty}}\left\Vert f_{2}\right\Vert _{L^{\infty}%
}\left\Vert f_{3}\right\Vert _{L^{\infty}}\lesssim2^{\varepsilon^{\prime}s}\ ,
\end{align*}
for all $\nu^{\prime}$-disjoint families of $2^{-s}$-separatede $2^{-s}%
$-tubes, and hence by rescaling that%
\begin{align*}
&  \left\Vert \prod_{k=1}^{3}\mathcal{S}\mathbb{T}_{k}\left(  2^{-s}%
\times2^{-s}\times1\right)  \right\Vert _{L^{\frac{q_{\varepsilon}}{3}}\left(
\mathbb{R}^{3}\right)  }=\left(  \int_{\mathbb{R}^{3}}\left(  \prod_{k=1}%
^{3}\mathcal{S}\mathbb{T}_{k}\left(  2^{-s}\times2^{-s}\times1\right)  \left(
\xi\right)  \right)  ^{\frac{q_{\varepsilon}}{3}}d\xi\right)  ^{\frac
{3}{q_{\varepsilon}}}\\
&  =\left(  \int_{\mathbb{R}^{3}}\left(  \prod_{k=1}^{3}\mathcal{S}%
\mathbb{T}_{k}\left(  2^{-s}\times2^{-s}\times1\right)  \left(  \frac
{\xi^{\prime}}{2^{2s}}\right)  \right)  ^{\frac{q_{\varepsilon}}{3}}d\frac
{\xi^{\prime}}{2^{2s}}\right)  ^{\frac{3}{q_{\varepsilon}}}%
\ \ \ \ \ \text{(with }\xi=\frac{\xi^{\prime}}{2^{2s}}\text{)}\\
&  =\left(  \int_{\mathbb{R}^{3}}\left(  \prod_{k=1}^{3}\mathcal{S}%
\mathbb{T}_{k}\left(  2^{s}\times2^{s}\times2^{2s}\right)  \left(  \xi
^{\prime}\right)  \right)  ^{\frac{q_{\varepsilon}}{3}}2^{-6s}d\xi^{\prime
}\right)  ^{\frac{3}{q_{\varepsilon}}}\\
&  =\left(  2^{-6s}\right)  ^{\frac{3}{q_{\varepsilon}}}\left\Vert \prod
_{k=1}^{3}\mathcal{S}\mathbb{T}_{k}\left(  2^{s}\times2^{s}\times
2^{2s}\right)  \right\Vert _{L^{\frac{q_{\varepsilon}}{3}}\left(
\mathbb{R}^{3}\right)  }\lesssim2^{\left(  6-\frac{18}{q_{\varepsilon}%
}\right)  s}2^{\varepsilon^{\prime}s}<2^{\left(  6-\frac{18}{3+\varepsilon
}+\varepsilon^{\prime}\right)  s}=2^{\left(  \frac{6\varepsilon}%
{3+\varepsilon}+\varepsilon^{\prime}\right)  s}<2^{3\varepsilon s},
\end{align*}
if we take $\varepsilon^{\prime}=\varepsilon$, since $q_{\varepsilon
}<3+\varepsilon$.

Arguing as in (\ref{interp}) we now obtain (\ref{tri Kak dual}) for this
choice of $\nu^{\prime}$. Indeed, with
\[
F\left(  \xi\right)  \equiv\prod_{k=1}^{3}\left(  \sum_{T_{k}\in\mathbb{T}%
_{k}\left(  2^{-s}\times2^{-s}\times1\right)  }\mathbf{1}_{T_{k}}\left(
\xi\right)  \right)  ^{\frac{1}{3}},
\]
we have using $q_{\varepsilon}>3$ that
\begin{align*}
&  \int_{\mathbb{R}^{3}}\prod_{k=1}^{3}\left(  \sum_{T_{k}\in\mathbb{T}%
_{k}\left(  2^{-s}\times2^{-s}\times1\right)  }\mathbf{1}_{T_{k}}\left(
\xi\right)  \right)  ^{\frac{1}{2}}d\xi=\int_{\mathbb{R}^{3}}F\left(
\xi\right)  ^{\frac{3}{2}}d\xi\\
&  =\left\{  \int_{\left\{  F\leq1\right\}  }+\int_{\left\{  F>1\right\}
}\right\}  F\left(  \xi\right)  ^{\frac{3}{2}}d\xi\leq\int_{\left\{
F\leq1\right\}  }F\left(  \xi\right)  d\xi+\int_{\left\{  F>1\right\}
}F\left(  \xi\right)  ^{\frac{q_{\varepsilon}}{2}}d\xi\\
&  \lesssim\int_{\left\{  \sum_{T\in\mathbb{T}}\mathbf{1}_{T}\leq1\right\}
}\prod_{k=1}^{3}\left(  \sum_{T_{k}\in\mathbb{T}_{k}\left(  2^{-s}\times
2^{-s}\times1\right)  }\mathbf{1}_{T_{k}}\left(  \xi\right)  \right)
^{\frac{1}{3}}d\xi+\int_{\left\{  \sum_{T\in\mathbb{T}}\mathbf{1}%
_{T}>1\right\}  }\prod_{k=1}^{3}\left(  \sum_{T_{k}\in\mathbb{T}_{k}\left(
2^{-s}\times2^{-s}\times1\right)  }\mathbf{1}_{T_{k}}\left(  \xi\right)
\right)  ^{\frac{q_{\varepsilon}}{6}}d\xi\\
&  \lesssim1+C_{q_{\varepsilon}}^{\frac{q_{\varepsilon}}{3}}2^{\frac
{q_{\varepsilon}}{3}3\varepsilon s}\lesssim1+2^{4\varepsilon s},
\end{align*}
since first,%
\begin{align*}
&  \int_{\left\{  \sum_{T\in\mathbb{T}}\mathbf{1}_{T}\leq1\right\}  }%
\prod_{k=1}^{3}\left(  \sum_{T_{k}\in\mathbb{T}_{k}\left(  2^{-s}\times
2^{-s}\times1\right)  }\mathbf{1}_{T_{k}}\left(  \xi\right)  \right)
^{\frac{1}{3}}d\xi\\
&  \lesssim\prod_{k=1}^{3}\left(  \int_{\mathbb{R}^{3}}\sum_{T_{k}%
\in\mathbb{T}_{k}\left(  2^{-s}\times2^{-s}\times1\right)  }\mathbf{1}_{T_{k}%
}\left(  \xi\right)  d\xi\right)  ^{\frac{1}{3}}\lesssim\prod_{k=1}^{3}\left(
2^{-2s}\#\mathbb{T}_{k}\left(  2^{-s}\times2^{-s}\times1\right)  \right)
^{\frac{1}{3}}\lesssim1,
\end{align*}
because the tubes in $\mathbb{T}_{k}$ are $2^{-s}$-separated, and second since
for $0<\varepsilon<1$, we have $q_{\varepsilon}<3+\varepsilon\leq4$ and,
\[
\int_{\left\{  \sum_{T\in\mathbb{T}}\mathbf{1}_{T}>1\right\}  }\prod_{k=1}%
^{3}\left(  \sum_{T_{k}\in\mathbb{T}_{k}\left(  2^{-s}\times2^{-s}%
\times1\right)  }\mathbf{1}_{T_{k}}\left(  \xi\right)  \right)  ^{\frac
{q_{\varepsilon}}{6}}d\xi\leq\int_{\mathbb{R}^{3}}\prod_{k=1}^{3}%
\mathcal{S}\mathbb{T}_{k}\left(  \xi\right)  ^{\frac{q_{\varepsilon}}{3}}%
d\xi\lesssim\left(  C_{q_{\varepsilon}}2^{3\varepsilon s}\right)
^{\frac{q_{\varepsilon}}{3}}\lesssim2^{4\varepsilon s}.
\]

Conversely, for the `only if' assertion, we use the rapid decay of the Fourier
transform away from the associated tubes, just as in the proof (\ref{rap}) of
the converse assertion of Proposition \ref{Kak equiv}, to obtain the
inequality,%
\begin{align*}
&  \left\Vert \prod_{k=1}^{3}\mathcal{S}_{\operatorname*{Fourier}%
}^{s,\mathbf{u}_{k}}f\right\Vert _{L^{1}\left(  \mathbb{R}^{3}\right)
}\lesssim\left\Vert \left(  \prod_{k=1}^{3}\sum_{m_{k}=0}^{\infty}a_{m_{k}%
}\sum_{I_{k}\in\mathcal{G}_{s}\left[  U_{k}\right]  }\mathbf{1}_{2^{m_{k}}%
\tau_{\mathbf{u}_{k}}\widehat{I}}\right)  ^{\frac{1}{2}}\right\Vert
_{L^{1}\left(  \mathbb{R}^{3}\right)  }\\
&  \lesssim\left\Vert \left(  \sum_{m_{1}=0}^{\infty}\sum_{m_{2}=0}^{\infty
}\sum_{m_{3}=0}^{\infty}a_{m_{1}}a_{m_{2}}a_{m_{3}}\prod_{k=1}^{3}\sum
_{I_{k}\in\mathcal{G}_{s}\left[  U_{k}\right]  }\mathbf{1}_{2^{m_{k}}%
\tau_{\mathbf{u}_{k}}\widehat{I}}\right)  ^{\frac{1}{2}}\right\Vert
_{L^{1}\left(  \mathbb{R}^{3}\right)  }\\
&  \lesssim\left\Vert \sum_{m_{1}=0}^{\infty}\sum_{m_{2}=0}^{\infty}%
\sum_{m_{3}=0}^{\infty}\sqrt{a_{m_{1}}a_{m_{2}}a_{m_{3}}}\prod_{k=1}%
^{3}\left(  \sum_{I_{k}\in\mathcal{G}_{s}\left[  U_{k}\right]  }%
\mathbf{1}_{2^{m_{k}}\tau_{\mathbf{u}_{k}}\widehat{I}}\right)  ^{\frac{1}{2}%
}\right\Vert _{L^{1}\left(  \mathbb{R}^{3}\right)  }\\
&  \lesssim\sum_{m_{1}=0}^{\infty}\sum_{m_{2}=0}^{\infty}\sum_{m_{3}%
=0}^{\infty}\sqrt{a_{m_{1}}a_{m_{2}}a_{m_{3}}}\left\Vert \prod_{k=1}%
^{3}\left(  \sum_{I_{k}\in\mathcal{G}_{s}\left[  U\right]  }\mathbf{1}%
_{2^{m_{k}}\tau_{\mathbf{u}_{k}}\widehat{I}}\right)  ^{\frac{1}{2}}\right\Vert
_{L^{1}\left(  \mathbb{R}^{3}\right)  }\\
&  \lesssim\sum_{m_{1}=0}^{\infty}\sum_{m_{2}=0}^{\infty}\sum_{m_{3}%
=0}^{\infty}\sqrt{a_{m_{1}}a_{m_{2}}a_{m_{3}}}C_{0}^{m_{1}+m_{2}+m_{3}%
}\left\Vert \prod_{k=1}^{3}\left(  \sum_{I_{k}\in\mathcal{G}_{s}\left[
U\right]  }\mathbf{1}_{\tau_{\mathbf{u}_{k}}\widehat{I}}\right)  ^{\frac{1}%
{2}}\right\Vert _{L^{1}\left(  \mathbb{R}^{3}\right)  }\lesssim2^{\varepsilon
s},
\end{align*}
since the coefficients $a_{m}$ are rapidly decreasing, and where again, the
constant $C_{0}$ is related to the geometry of expanded tubes. This shows that
(\ref{single tri Four}) holds for all $\varepsilon,\nu>0$, and completes the
proof of Proposition \ref{Kak tri equiv}.
\end{proof}

Combining Proposition \ref{Kak equiv}, Theorem \ref{SFA} and Proposition
\ref{Kak tri equiv}, yields Theorem \ref{big}.

\begin{remark}
The dual form of the Kakeya maximal operator conjecture, as well as its
triliner analogue, can be reduced to proving the case where all of the
$\delta$-tubes of length $1$ are contained in the cube $\left[  -2,2\right]
^{3}$. Indeed, tile $\mathbb{R}^{3}$ with translates $Q_{\alpha}\equiv
\alpha+\left[  -2,2\right]  ^{3}$ of the cube $\left[  -2,2\right]  ^{3}$ for
$\alpha\in\left(  2\mathbb{Z}\right)  ^{3}$, and set $\mathbb{T}_{\alpha
}\equiv\left\{  T\in\mathbb{T}:T\subset2Q_{\alpha}\right\}  $. Then every
$\delta$-tube $T$ belongs to at least one, and at most nine, of the
collections $\mathbb{T}_{\alpha}$. Suppose that for some $p>0$, and for every
$\alpha\in\left(  2\mathbb{Z}\right)  ^{3}$, we have the inequality (the case
$\alpha$ is equivalent to the case $0$ by translation invariance),%
\[
\int_{\mathbb{R}^{3}}\left(  \sum_{T\in\mathbb{T}_{\alpha}}\mathbf{1}%
_{T}\left(  \xi\right)  \right)  ^{p}d\xi\lesssim\delta^{-\varepsilon}\left(
\delta^{2}\#\mathbb{T}_{\alpha}\right)  .
\]
Then from the bounded overlap of the cubes $\left\{  2Q_{\alpha}\right\}
_{\alpha\in\left(  2\mathbb{Z}\right)  ^{3}}$ we obtain,
\begin{align*}
&  \int_{\mathbb{R}^{3}}\left(  \sum_{T\in\mathbb{T}}\mathbf{1}_{T}\left(
\xi\right)  \right)  ^{p}d\xi=\int_{\mathbb{R}^{3}}\left(  \sum_{\alpha
\in\left(  2\mathbb{Z}\right)  ^{3}}\sum_{T\in\mathbb{T}_{\alpha}}%
\mathbf{1}_{T}\left(  \xi\right)  \right)  ^{p}d\xi\\
&  \approx\int_{\mathbb{R}^{3}}\sum_{\alpha\in\left(  2\mathbb{Z}\right)
^{3}}\left(  \sum_{T\in\mathbb{T}_{\alpha}}\mathbf{1}_{T}\left(  \xi\right)
\right)  ^{p}d\xi\lesssim\sum_{\alpha\in\left(  2\mathbb{Z}\right)  ^{3}%
}\delta^{-\varepsilon}\left(  \delta^{2}\#\mathbb{T}_{\alpha}\right)
\leq9\delta^{-\varepsilon}\left(  \delta^{2}\#\mathbb{T}\right)
\lesssim\delta^{-\varepsilon}.
\end{align*}
A similar result holds for the trilinear version of the Kakeya maximal
operator conjecture upon using%
\[
\prod_{k=1}^{3}\left(  \sum_{T_{k}\in\mathbb{T}_{k}}\mathbf{1}_{T_{k}}\left(
\xi\right)  \right)  ^{\frac{p}{3}}=\prod_{k=1}^{3}\left(  \sum_{\alpha_{k}%
\in\left(  2\mathbb{Z}\right)  ^{3}}\sum_{T_{k}\in\left(  \mathbb{T}%
_{k}\right)  _{\alpha_{k}}}\mathbf{1}_{T_{k}}\left(  \xi\right)  \right)
^{\frac{p}{3}}\approx\sum_{\alpha\in\left(  2\mathbb{Z}\right)  ^{3}}%
\prod_{k=1}^{3}\left(  \sum_{T_{k}\in\left(  \mathbb{T}_{k}\right)  _{\alpha}%
}\mathbf{1}_{T_{k}}\left(  \xi\right)  \right)  ^{\frac{p}{3}},
\]
which again follows from the bounded overlap of the cubes $\left\{
2Q_{\alpha}\right\}  _{\alpha\in\left(  2\mathbb{Z}\right)  ^{3}}$, since if
$\xi\in2Q_{\alpha}$, then we must have $\left\vert \alpha_{k}-\alpha
\right\vert \leq2$ in order that the product above is nonvanishing.
\end{remark}

\section{A detailed proof of the square function theorem}

Here we write out complete details of the proof of Theorem \ref{SFA}, which
broadly follows the argument in \cite[proof of Theorem 3]{RiSa} - in turn
built on the argument in \cite[Section 2]{BoGu} - but is here adapted to the
use of square functions, modulations and single scale inequalities. As
mentioned earlier, we do not need the induction on scales idea used in the
difficult \textbf{Case 3} of \cite[Section 2]{BoGu}.

Suppose $S$ is a compact smooth hypersurface in $\mathbb{R}^{3}$ that is
contained in the paraboloid $\mathbb{P}^{2}$, and denote surface measure on
$S$ by $\sigma$. Define%
\[
\mathcal{S}_{\operatorname*{Fourier}}^{s,\mathbf{u}}f\left(  \xi\right)
\equiv\left(  \sum_{L\in\mathcal{G}_{s}\left[  U\right]  }\left\vert \left[
\mathsf{M}_{\mathbf{u}}^{s}\Phi_{\ast}\bigtriangleup_{L}f\right]  ^{\wedge
}\left(  \xi\right)  \right\vert ^{2}\right)  ^{\frac{1}{2}}.
\]
The next definition is specialized from \cite{BoGu} to accommodate the
\emph{single scale} inequality for $s\in\mathbb{N}$ that is used here, in
place of the \emph{local} inequalities for balls $B_{R}$ that are used in
\cite{BoGu} and \cite{RiSa}.

\begin{definition}
For $1<q<\infty$ and $s\in\mathbb{N}$ define $A_{s}^{\left(  q\right)  }$ to
be the best constant in the single scale linear Fourier extension inequality,
\[
\left(  \int_{\mathbb{R}^{3}}\left\vert \mathcal{S}_{\operatorname*{Fourier}%
}^{s,\mathbf{u}}f\left(  \xi\right)  \right\vert ^{q}d\xi\right)  ^{\frac
{1}{q}}\leq A_{s}^{\left(  q\right)  }\left\Vert f\right\Vert _{L^{\infty
}\left(  U\right)  }\ ,\ \ \ \ \ \text{for all grids }\mathcal{G}\text{,
}\mathbf{u}\in\mathcal{V}\text{ and }f\in L^{\infty}\left(  U\right)  ,
\]
i.e.
\begin{equation}
A_{s}^{\left(  q\right)  }\equiv\sup_{\mathcal{G}\in\operatorname*{Grid}}%
\sup_{\mathbf{u}\in\mathcal{V}}\sup_{\left\Vert f\right\Vert _{L^{\infty
}\left(  U\right)  }\leq1}\left(  \int_{\mathbb{R}^{3}}\left\vert
\mathcal{S}_{\operatorname*{Fourier}}^{s,\mathbf{u}}f\left(  \xi\right)
\right\vert ^{q}d\xi\right)  ^{\frac{1}{q}}. \label{Q_R}%
\end{equation}

\end{definition}

Note that for each fixed $q>3$ and $s\in\mathbb{N}$, the quantity
$A_{s}^{\left(  q\right)  }$ is finite since $\Phi_{\ast}\bigtriangleup_{L}f$
is a smooth measure on the paraboloid and so $\left[  \mathsf{M}_{\mathbf{u}%
}^{s}\Phi_{\ast}\bigtriangleup_{L}f\right]  ^{\wedge}\left(  \xi\right)
=\left[  \Phi_{\ast}\bigtriangleup_{L}f\right]  ^{\wedge}\left(  \xi
-u_{L}\right)  $ is in $L^{q}\left(  \mathbb{R}^{3}\right)  $ for $q>3$
uniformly in $\mathcal{G}$, $\mathbf{u}$ and $f\in L^{\infty}\left(  U\right)
$.

\begin{theorem}
\label{Loc lin}Let $S$ be as above and suppose that for all $q>3$, there is
$0<\nu<1$ such that $\mathcal{A}_{\operatorname*{disj}\nu}%
^{\operatorname*{square}}\left(  \otimes_{3}L^{\infty}\rightarrow L^{\frac
{q}{3}};\varepsilon\right)  $ holds for all $0<\varepsilon<1$. Then for every
$\varepsilon>0$ and $s\in\mathbb{N}$, we have%
\[
A_{s}^{\left(  q\right)  }\leq C_{q}2^{\frac{\varepsilon}{3}s}.
\]

\end{theorem}

We first introduce a shorter notation for the singular measures on the
paraboloid of the form
\[
\mathsf{M}_{\mathbf{u}}^{s}\Phi_{\ast}\sum_{L\in\mathcal{G}_{s}\left[
I\right]  }\bigtriangleup_{L}f=\mathsf{M}_{\mathbf{u}}^{s}\Phi_{\ast
}\mathsf{Q}_{I}^{s}f,\ \ \ \ \ \text{where }\mathsf{Q}_{I}^{s}\equiv\sum
_{L\in\mathcal{G}_{s}\left[  I\right]  }\bigtriangleup_{L}f,
\]
that are used repeatedly throughout the proof. For $\lambda\in\mathbb{N}$,
$I\in\mathcal{G}_{\lambda}\left[  U\right]  $ and $s\geq\lambda\,$, we set%
\begin{equation}
f_{I}^{\Phi,s}\equiv\Phi_{\ast}\mathsf{Q}_{I}^{s}f=\Phi_{\ast}\left(
\mathsf{Q}_{I}^{s}f\left(  x\right)  dx\right)  \text{ and }\mathsf{Q}_{I}%
^{s}f\left(  x\right)  \equiv\sum_{L\in\mathcal{G}_{s}\left[  I\right]
}\bigtriangleup_{L}f\left(  x\right)  =\sum_{L\in\mathcal{G}_{s}\left[
I\right]  }\left\langle f,\varphi_{L}\right\rangle \varphi_{L}\left(
x\right)  , \label{def delta}%
\end{equation}
and%
\begin{equation}
f_{I,\mathbf{u}}^{\Phi,s}\equiv\mathsf{M}_{\mathbf{u}}^{s}\Phi_{\ast
}\mathsf{Q}_{I}^{s}f=\mathsf{M}_{\mathbf{u}}^{s}f_{I}^{\Phi}%
,\ \ \ \ \ \text{for }I\in\mathcal{G}_{\lambda}\left[  U\right]  \text{ and
}s\geq\lambda, \label{def f I u}%
\end{equation}
where $\mathsf{M}_{\mathbf{u}}^{s}$ is defined in Definition \ref{def M}. We
also define the collection $\mathcal{F}$ of singular measures on the
paraboloid by,
\[
\mathcal{F}\equiv\left\{  f_{I,\mathbf{u}}^{\Phi,s}:I\in\mathcal{G}_{\lambda
}\left[  U\right]  \text{, }\mathcal{G}\in\operatorname*{Grid}\text{, }%
s\in\mathbb{N}\text{, and }\mathbf{u}\in\mathcal{V}\right\}  ,
\]
which is invariant under parabolic rescaling by a dyadic number, see
(\ref{par res}) below and the subsequent discussion.

\subsection{The pigeonholing argument of Bourgain and Guth\label{Sub pigeon}}

\begin{proof}
[Proof of Theorem \ref{Loc lin}]We begin the argument as in \cite{BoGu}, with
notation as in \cite{RiSa}, but with the added complications of having to deal
with modulations $\mathsf{M}_{\mathbf{u}}^{s}$ and Fourier square functions
$\mathcal{S}_{\operatorname*{Fourier}}^{s,\mathbf{u}}$. For $S$ a compact
smooth surface contained in the paraboloid $\mathbb{P}^{2}$ given by
$z_{3}=\left\vert z^{\prime}\right\vert ^{2}=z_{1}^{2}+z_{2}^{2}$ in
$\mathbb{R}^{3}$, and for $f\in L^{\infty}\left(  S\right)  $ with $\left\Vert
f\right\Vert _{L^{\infty}\left(  S\right)  }=1$, and for $\mathcal{G}%
\in\operatorname*{Grid}$, we initially consider the oscillatory integral,%
\begin{align*}
&  Tf\left(  \xi\right)  =\int_{U}e^{i\phi\left(  \xi,y\right)  }\sum
_{I\in\mathcal{G}_{\lambda}\left[  U\right]  }\mathsf{Q}_{I}^{s}f\left(
y\right)  dy=\sum_{I\in\mathcal{G}_{\lambda}\left[  U\right]  }\int
_{U}e^{i\left\{  \xi_{1}\cdot y_{1}+\xi_{2}\cdot y_{2}+\xi_{3}\left(
y_{1}^{2}+y_{2}^{2}\right)  \right\}  }\mathsf{Q}_{I}^{s}f\left(  y\right)
dy\\
&  =\sum_{I\in\mathcal{G}_{\lambda}\left[  U\right]  }\int_{U}e^{i\xi
\cdot\left(  y,\left\vert y\right\vert ^{2}\right)  }\mathsf{Q}_{I}%
^{s}f\left(  y\right)  dy=\sum_{I\in\mathcal{G}_{\lambda}\left[  U\right]
}\widehat{f_{I}^{\Phi,s}}\left(  \xi\right)  =\sum_{I\in\mathcal{G}_{\lambda
}\left[  U\right]  }\left(  \Phi_{\ast}\left[  \mathsf{Q}_{I}^{s}f\left(
y\right)  dy\right]  \right)  ^{\wedge}\left(  \xi\right)
,\ \ \ \ \ \text{for }\xi\in\mathbb{R}^{3},
\end{align*}
where $\bigtriangleup_{I}f$ is defined in (\ref{def delta}), and
\[
\phi\left(  \xi,y\right)  =\xi\cdot\Phi\left(  y\right)  \text{ and }%
\Phi\left(  y\right)  \equiv\left(  y_{1},y_{2},y_{1}^{2}+y_{2}^{2}\right)  .
\]
We write%
\begin{align*}
Tf\left(  \xi\right)   &  =\sum_{I\in\mathcal{G}_{\lambda}\left[  U\right]
}\int_{U}e^{i\xi\cdot\left(  y,\left\vert y\right\vert ^{2}\right)
}\mathsf{Q}_{I}^{s}f\left(  y\right)  dy\\
&  =\sum_{I\in\mathcal{G}_{\lambda}\left[  U\right]  }e^{i\phi\left(
\xi,c_{I}\right)  }\int e^{i\left\{  \phi\left(  \xi,y\right)  -\phi\left(
\xi,c_{I}\right)  \right\}  }\mathsf{Q}_{I}^{s}f\left(  y\right)
dy=\sum_{I\in\mathcal{G}_{\lambda}\left[  S\right]  }e^{i\phi\left(  \xi
,c_{I}\right)  }T_{I}f\left(  \xi\right)  ,
\end{align*}
where%
\begin{align}
T_{I}f\left(  \xi\right)   &  \equiv\int e^{i\left\{  \phi\left(
\xi,y\right)  -\phi\left(  \xi,c_{I}\right)  \right\}  }\mathsf{Q}_{I}%
^{s}f\left(  y\right)  dy=e^{-i\xi\cdot\Phi\left(  c_{I}\right)  }\int
e^{i\xi\cdot\Phi\left(  y\right)  }\mathsf{Q}_{I}^{s}f\left(  y\right)
dy\label{def T_I}\\
&  =e^{-i\xi\cdot\Phi\left(  c_{I}\right)  }\widehat{f_{I}^{\Phi,s}}\left(
\xi\right)  =\left(  \tau_{-\Phi\left(  c_{I}\right)  }f_{I}^{\Phi,s}\right)
^{\wedge}\left(  \xi\right)  ,\nonumber
\end{align}
and where $\tau_{\Phi\left(  c_{I}\right)  }g\left(  z\right)  \equiv g\left(
z-\Phi\left(  c_{I}\right)  \right)  $ is translation of a function $g$ by the
unit vector $\Phi\left(  c_{I}\right)  $.

The \emph{Fourier square function} at level $s$ is defined by
\begin{align}
\mathcal{S}_{\operatorname*{Fourier}}^{s}f\left(  \xi\right)   &
\equiv\left(  \sum_{L\in\mathcal{G}_{s}\left[  U\right]  }\left\vert
\widehat{f_{L}^{\Phi,s}}\left(  \xi\right)  \right\vert ^{2}\right)
^{\frac{1}{2}}=\left(  \sum_{L\in\mathcal{G}_{s}\left[  U\right]  }\left\vert
\widehat{\Phi_{\ast}\bigtriangleup_{L}f}\left(  \xi\right)  \right\vert
^{2}\right)  ^{\frac{1}{2}}\label{w/o}\\
&  =\left(  \sum_{I\in\mathcal{G}_{\lambda}\left[  U\right]  }\sum
_{L\in\mathcal{G}_{s}\left[  I\right]  }\left\vert \widehat{\Phi_{\ast
}\bigtriangleup_{L}f}\left(  \xi\right)  \right\vert ^{2}\right)  ^{\frac
{1}{2}}=\frac{1}{c_{\flat}}\left(  \sum_{I\in\mathcal{G}_{\lambda}\left[
U\right]  }\mathcal{S}_{\operatorname*{Fourier}}^{s}\mathsf{Q}_{I}^{s}f\left(
\xi\right)  ^{2}\right)  ^{\frac{1}{2}},\nonumber
\end{align}
for $s\geq\lambda$ since
\begin{align*}
&  \mathcal{S}_{\operatorname*{Fourier}}^{s}\mathsf{Q}_{I}^{s}f\left(
\xi\right)  =\left(  \sum_{L\in\mathcal{G}_{s}\left[  U\right]  }\left\vert
\left(  \Phi_{\ast}\bigtriangleup_{L}\mathsf{Q}_{I}^{s}f\right)  ^{\wedge
}\left(  \xi\right)  \right\vert ^{2}\right)  ^{\frac{1}{2}}\\
&  =\left(  \sum_{L\in\mathcal{G}_{s}\left[  U\right]  }\left\vert \sum
_{M\in\mathcal{G}_{s}\left[  I\right]  }\left(  \Phi_{\ast}\bigtriangleup
_{L}\bigtriangleup_{M}f\right)  ^{\wedge}\left(  \xi\right)  \right\vert
^{2}\right)  ^{\frac{1}{2}}=c_{\flat}\left(  \sum_{L\in\mathcal{G}_{s}\left[
I\right]  }\left\vert \widehat{\Phi_{\ast}\bigtriangleup_{L}f}\left(
\xi\right)  \right\vert ^{2}\right)  ^{\frac{1}{2}},
\end{align*}
where $c_{\flat}$ is defined in (\ref{pseudo}).

\medskip

More generally, in our situation where modulations $\mathsf{M}_{\mathbf{u}%
}^{s}$ multiply the pushforwards $\Phi_{\ast}\bigtriangleup_{L}f$, we first
note that the following intertwining formula holds,
\begin{align}
\mathsf{M}_{\mathbf{u}}^{s}\Phi_{\ast}\mathsf{Q}^{s}f\left(  z\right)   &
=\Phi_{\ast}\mathsf{Q}_{\mathbf{u}}^{s}f\left(  z\right)  ,\label{comm}\\
\text{where }\mathsf{Q}_{\mathbf{u}}^{s}f\left(  y\right)   &  \equiv
\sum_{L\in\mathcal{G}_{s}\left[  U\right]  }e^{iu_{L}\cdot\Phi\left(
y\right)  }\bigtriangleup_{L}f\left(  y\right)  .\nonumber
\end{align}
Indeed, for any continuous function $g\in C\left(  \mathbb{R}^{3}\right)  $,
we have by definition that%
\begin{align*}
&  \left\langle \mathsf{M}_{\mathbf{u}}^{s}\Phi_{\ast}\mathsf{Q}%
^{s}f,g\right\rangle =\int_{\mathbb{R}^{3}}g\left(  z\right)  d\left(
\mathsf{M}_{\mathbf{u}}^{s}\Phi_{\ast}\mathsf{Q}^{s}f\right)  \left(
z\right)  =\int_{\mathbb{R}^{3}}g\left(  z\right)  \mathsf{M}_{\mathbf{u}}%
^{s}\left(  z\right)  d\left(  \Phi_{\ast}\mathsf{Q}^{s}f\right)  \left(
z\right) \\
&  =\int_{\mathbb{R}^{3}}g\left(  z\right)  \sum_{L\in\mathcal{G}_{s}\left[
U\right]  }\mathbf{1}_{L_{s}}\left(  \Phi^{-1}z\right)  e^{iu_{L}\cdot
z}d\left(  \Phi_{\ast}\mathsf{Q}^{s}f\right)  \left(  z\right) \\
&  =\int_{U}\left\{  g\left(  \Phi\left(  y\right)  \right)  \sum
_{L\in\mathcal{G}_{s}\left[  U\right]  }\mathbf{1}_{L_{s}}\left(  y\right)
e^{iu_{L}\cdot\Phi\left(  y\right)  }\right\}  \mathsf{Q}^{s}f\left(
y\right)  dy=\int_{U}g\left(  \Phi\left(  y\right)  \right)  \left\{
\sum_{L\in\mathcal{G}_{s}\left[  U\right]  }e^{iu_{L}\cdot\Phi\left(
y\right)  }\bigtriangleup_{L}f\left(  y\right)  \right\}  dy
\end{align*}
where $L_{s}\equiv\left\{  y\in\mathbb{R}^{2}:\operatorname*{dist}\left(
y,L\right)  <2^{-s}\right\}  $, and%
\[
\left\langle \Phi_{\ast}\mathsf{Q}_{\mathbf{u}}^{s}f\left(  z\right)
,g\right\rangle =\int_{\mathbb{R}^{3}}g\left(  z\right)  d\left(  \Phi_{\ast
}\mathsf{Q}_{\mathbf{u}}^{s}f\right)  \left(  z\right)  =\int_{U}g\left(
\Phi\left(  y\right)  \right)  \left\{  \mathsf{Q}_{\mathbf{u}}^{s}f\left(
y\right)  \right\}  dy,
\]
which proves (\ref{comm}).

Now we continue by defining a localization of $\mathsf{Q}_{\mathbf{u}}^{s}$ to
a square $I$,
\[
\mathsf{Q}_{I,\mathbf{u}}^{s}f\left(  y\right)  \equiv\sum_{L\in
\mathcal{G}_{s}\left[  I\right]  }e^{iu_{L}\cdot\Phi\left(  y\right)
}\bigtriangleup_{L}f\left(  y\right)  =\mathsf{Q}_{\mathbf{u}}^{s}%
\mathbf{1}_{I_{s}}f\left(  y\right)  ,\ \ \ \ \ \text{for }I\in\mathcal{G}%
_{\lambda}\left[  U\right]  \text{ and }\lambda\leq s,
\]
where $I_{s}\equiv\left\{  y\in\mathbb{R}^{2}:\operatorname*{dist}\left(
y,I\right)  <2^{-s}\right\}  $, so that $\mathsf{Q}_{\mathbf{u}}^{s}f\left(
y\right)  =\sum_{I\in\mathcal{G}_{\lambda}\left[  U\right]  }\mathsf{Q}%
_{I,\mathbf{u}}^{s}f\left(  y\right)  $. Define for each $\mathbf{u}=\left\{
u_{L}\right\}  _{L\in\mathcal{G}_{s}\left[  U\right]  }\in\mathcal{V}$,%
\begin{align*}
T_{\mathbf{u}}f\left(  \xi\right)   &  =\sum_{I\in\mathcal{G}_{\lambda}\left[
U\right]  }\int_{U}e^{i\xi\cdot\left(  y,\left\vert y\right\vert ^{2}\right)
}\mathsf{Q}_{I,\mathbf{u}}^{s}f\left(  y\right)  dy\\
&  =\sum_{I\in\mathcal{G}_{\lambda}\left[  U\right]  }e^{i\phi\left(
\xi,c_{I}\right)  }\int e^{i\left\{  \phi\left(  \xi,y\right)  -\phi\left(
\xi,c_{I}\right)  \right\}  }\mathsf{Q}_{I,\mathbf{u}}^{s}f\left(  y\right)
dy=\sum_{I\in\mathcal{G}_{\lambda}\left[  S\right]  }e^{i\phi\left(  \xi
,c_{I}\right)  }T_{I,\mathbf{u}}f\left(  \xi\right)  ,
\end{align*}
where%
\begin{align*}
T_{I,\mathbf{u}}f\left(  \xi\right)   &  \equiv\int e^{i\left\{  \phi\left(
\xi,y\right)  -\phi\left(  \xi,c_{I}\right)  \right\}  }\mathsf{Q}%
_{I,\mathbf{u}}^{s}f\left(  y\right)  dy=e^{-i\xi\cdot\Phi\left(
c_{I}\right)  }\int e^{i\xi\cdot\Phi\left(  y\right)  }\mathsf{Q}%
_{I,\mathbf{u}}^{s}f\left(  y\right)  dy\\
&  =e^{-i\xi\cdot\Phi\left(  c_{I}\right)  }\widehat{\Phi_{\ast}%
\mathsf{Q}_{I,\mathbf{u}}^{s}f}\left(  \xi\right)  =e^{-i\xi\cdot\Phi\left(
c_{I}\right)  }\widehat{f_{I,\mathbf{u}}^{\Phi,s}}\left(  \xi\right)
=\widehat{\tau_{-\Phi\left(  c_{I}\right)  }f_{I,\mathbf{u}}^{\Phi,s}}\left(
\xi\right)  ,
\end{align*}
since%
\begin{align*}
\Phi_{\ast}\mathsf{Q}_{I,\mathbf{u}}^{s}f\left(  z\right)   &  =\Phi_{\ast
}\left(  \sum_{L\in\mathcal{G}_{s}\left[  I\right]  }e^{iu_{L}\cdot\Phi\left(
y\right)  }\bigtriangleup_{L}f\left(  y\right)  dy\right)  \left(  z\right)
=\sum_{L\in\mathcal{G}_{s}\left[  I\right]  }e^{iu_{L}\cdot z}\left(
\Phi_{\ast}\bigtriangleup_{L}f\right)  \left(  z\right) \\
&  =\sum_{L\in\mathcal{G}_{s}\left[  I\right]  }\mathsf{M}_{\mathbf{u}}%
^{s}\left(  z\right)  \left(  \Phi_{\ast}\bigtriangleup_{L}f\right)  \left(
z\right)  =f_{I,\mathbf{u}}^{\Phi,s}\left(  z\right)  .
\end{align*}

The corresponding Fourier square function in analogy with (\ref{w/o}) is now,%
\begin{align*}
\mathcal{S}_{\operatorname*{Fourier}}^{s,\mathbf{u}}f\left(  \xi\right)   &
=\left(  \sum_{L\in\mathcal{G}_{s}\left[  U\right]  }\left\vert \widehat
{f_{L,\mathbf{u}}^{\Phi,s}}\left(  \xi\right)  \right\vert ^{2}\right)
^{\frac{1}{2}}\equiv\left(  \sum_{L\in\mathcal{G}_{s}\left[  U\right]
}\left\vert \left(  \Phi_{\ast}\mathsf{M}_{\mathbf{u}}^{s}\bigtriangleup
_{L}f\right)  ^{\wedge}\left(  \xi\right)  \right\vert ^{2}\right)  ^{\frac
{1}{2}}\\
&  =\left(  \sum_{I\in\mathcal{G}_{\lambda}\left[  U\right]  }\sum
_{L\in\mathcal{G}_{s}\left[  I\right]  }\left\vert \left(  \mathsf{M}%
_{\mathbf{u}}^{s}\Phi_{\ast}\bigtriangleup_{L}f\right)  ^{\wedge}\left(
\xi\right)  \right\vert ^{2}\right)  ^{\frac{1}{2}}=\left(  \sum
_{I\in\mathcal{G}_{\lambda}\left[  U\right]  }\mathcal{S}%
_{\operatorname*{Fourier}}^{s}\mathbf{1}_{I_{s}}f\left(  \xi\right)
^{2}\right)  ^{\frac{1}{2}}.
\end{align*}

Note that%
\[
\left\vert \nabla_{\xi}\left\{  \left(  \xi-u_{I}\right)  \cdot\left(
y-c_{I},\left\vert y\right\vert ^{2}-\left\vert c_{I}\right\vert ^{2}\right)
\right\}  \right\vert =\left\vert \left(  y-c_{I},\left\vert y\right\vert
^{2}-\left\vert c_{I}\right\vert ^{2}\right)  \right\vert \lesssim\frac
{1}{2^{\lambda}},\ \ \ \ \ \text{for }y\in I,
\]
implies%
\begin{align}
\nabla_{\xi}T_{I,\mathbf{u}}f\left(  \xi\right)   &  =\nabla_{\xi}\int
e^{i\xi\cdot\left(  y-c_{I},\left\vert y-c_{I}\right\vert ^{2}\right)
}\mathsf{Q}_{I,\mathbf{u}}^{s}f\left(  y\right)  dy=\int\nabla_{\xi}%
e^{i\xi\cdot\left(  y-c_{I},\left\vert y-c_{I}\right\vert ^{2}\right)
}\mathsf{Q}_{I,\mathbf{u}}^{s}f\left(  y\right)  dy\label{pre note that}\\
&  =\int ie^{i\xi\cdot\left(  y-c_{I},\left\vert y-c_{I}\right\vert
^{2}\right)  }\nabla_{\xi}\left\{  \xi\cdot\left(  y-c_{I},\left\vert
y\right\vert ^{2}-\left\vert c_{I}\right\vert ^{2}\right)  \right\}
\mathsf{Q}_{I,\mathbf{u}}^{s}f\left(  y\right)  dy,\nonumber
\end{align}
which implies%
\begin{equation}
\left\vert \nabla_{\xi}T_{I,\mathbf{u}}f\left(  \xi\right)  \right\vert
\leq\int\left\vert \nabla_{\xi}\left\{  \xi\cdot\left(  y-c_{I},\left\vert
y\right\vert ^{2}-\left\vert c_{I}\right\vert ^{2}\right)  \right\}
\right\vert \left\vert \mathsf{Q}_{I,\mathbf{u}}^{s}f\left(  y\right)
\right\vert dy\lesssim\frac{1}{2^{\lambda}}\left\Vert \mathsf{Q}%
_{I,\mathbf{u}}^{s}f\right\Vert _{L^{1}\left(  U\right)  }\lesssim\frac
{1}{2^{3\lambda}}\left\Vert \mathsf{Q}_{I,\mathbf{u}}^{s}f\right\Vert
_{L^{\infty}\left(  U\right)  }, \label{note that}%
\end{equation}
since $\ell\left(  I\right)  =\frac{1}{2^{\lambda}}$. We will use the
estimates (\ref{pre note that}) and (\ref{note that}) in (\ref{exclaim'}) below.

Now let $\rho$ be a smooth rapidly decreasing bump function such that
$\widehat{\rho}\left(  \xi\right)  =1$ for $\left\vert \xi\right\vert \leq1$,
and set%
\[
\rho_{\lambda}\left(  z\right)  \equiv\frac{1}{2^{3\lambda}}\rho\left(
\frac{z}{2^{\lambda}}\right)  ,\ \ \ \ \ \widehat{\rho_{\lambda}}\left(
\xi\right)  =\widehat{\rho}\left(  2^{\lambda}\xi\right)  =1\text{ on
}B\left(  0,2^{-\lambda}\right)  \text{ and }\rho_{\lambda}\left(  z\right)
\approx\frac{1}{2^{3\lambda}}\text{ on }B\left(  0,2^{\lambda}\right)  .
\]
Then from (\ref{def T_I}) we obtain%
\[
T_{I,\mathbf{u}}f\left(  \xi\right)  =T_{I,\mathbf{u}}f\ast\rho_{\lambda
}\left(  \xi\right)  \ ,\ \ \ \ \ \text{for }I\in\mathcal{G}_{\lambda}\left[
S\right]  \text{ and }\xi\in\mathbb{R}^{3},
\]
since $\tau_{-\Phi\left(  c_{I}\right)  }f_{I,\mathbf{u}}^{\Phi}\subset
B\left(  0,2^{-\lambda}\right)  $ and $\widehat{T_{I,\mathbf{u}}f}\left(
z\right)  =\tau_{-\Phi\left(  c_{I}\right)  }f_{I,\mathbf{u}}^{\Phi}\left(
z\right)  $ imply%
\[
\widehat{T_{I,\mathbf{u}}f\ast\rho_{\lambda}}\left(  z\right)  =\widehat
{T_{I,\mathbf{u}}f}\left(  z\right)  \widehat{\rho_{\lambda}}\left(  z\right)
=\tau_{-\Phi\left(  c_{I}\right)  }f_{I,\mathbf{u}}^{\Phi}\left(  z\right)
\widehat{\rho_{\lambda}}\left(  z\right)  =\tau_{-\Phi\left(  c_{I}\right)
}f_{I,\mathbf{u}}^{\Phi}\left(  z\right)  =\widehat{T_{I,\mathbf{u}}f}\left(
z\right)  .
\]

Now fix a point $a\in\Gamma_{\lambda}$, where%
\[
\Gamma_{\lambda}\equiv2^{\lambda}\mathbb{Z}^{3},
\]
and restrict $\xi$ to the ball $B\left(  a,2^{\lambda}\right)  $. Then for
$\xi\in B\left(  a,2^{\lambda}\right)  $ and $I\in\mathcal{G}_{\lambda}\left[
S\right]  $ we have
\begin{align}
&  \left\vert T_{I,\mathbf{u}}f\left(  \xi\right)  \right\vert =\left\vert
T_{I,\mathbf{u}}f\ast\rho_{\lambda}\left(  \xi\right)  \right\vert =\left\vert
\int_{\mathbb{R}^{3}}T_{I,\mathbf{u}}f\left(  z\right)  \rho_{\lambda}\left(
\xi-z\right)  dz\right\vert \label{exclaim}\\
&  \leq\int_{\mathbb{R}^{3}}\left\vert T_{I,\mathbf{u}}f\left(  z\right)
\right\vert \left\vert \rho_{\lambda}\left(  \xi-z\right)  \right\vert
dz\leq\int_{\mathbb{R}^{3}}\left\vert T_{I,\mathbf{u}}f\left(  z\right)
\right\vert \sup_{\omega\in B\left(  a,2^{\lambda}\right)  }\left\vert
\rho_{\lambda}\left(  z-\omega\right)  \right\vert dz=\int_{\mathbb{R}^{3}%
}\left\vert T_{I,\mathbf{u}}f\left(  z\right)  \right\vert \zeta_{\lambda
}\left(  z-a\right)  dz,\nonumber
\end{align}
where $\zeta_{\lambda}\left(  w\right)  \equiv\sup_{\omega-a\in B\left(
0,2^{\lambda}\right)  }\left\vert \rho_{\lambda}\left(  \omega\right)
\right\vert $, since $\rho$ can be chosen radial and,
\begin{align*}
\sup_{\omega\in B\left(  a,2^{\lambda}\right)  }\left\vert \rho_{\lambda
}\left(  \omega-z\right)  \right\vert  &  =\frac{1}{2^{3\lambda}}\sup
_{\omega\in B\left(  a,2^{\lambda}\right)  }\left\vert \rho\left(
\frac{\left(  \omega-a\right)  -\left(  z-a\right)  }{2^{\lambda}}\right)
\right\vert =\frac{1}{2^{3\lambda}}\sup_{\gamma\in B\left(  0,1\right)
}\left\vert \rho\left(  \frac{z-a}{2^{\lambda}}-\gamma\right)  \right\vert
=\zeta_{\lambda}\left(  z-a\right)  ,\\
\text{where }\zeta\left(  w\right)   &  \equiv\sup_{\left\vert w-w^{\prime
}\right\vert \leq1}\left\vert \rho\left(  w^{\prime}\right)  \right\vert .
\end{align*}

Moreover, for $L\in\mathcal{G}_{s}\left[  I\right]  $ and $\xi\in B\left(
a,2^{\lambda}\right)  $, we also have%
\[
\left\vert T_{L,\mathbf{u}}f\left(  \xi\right)  \right\vert \leq
\int_{\mathbb{R}^{3}}\left\vert T_{L,\mathbf{u}}f\left(  z\right)  \right\vert
\zeta_{\lambda}\left(  z-a\right)  dz,
\]
and the corresponding square function estimate%
\begin{equation}
\mathcal{S}_{\operatorname*{Fourier}}^{s}T_{I,\mathbf{u}}f\left(  \xi\right)
=\left(  \sum_{L\in\mathcal{G}_{s}\left[  I\right]  }\left\vert
T_{L,\mathbf{u}}f\left(  \xi\right)  \right\vert ^{2}\right)  ^{\frac{1}{2}%
}\leq\left(  \sum_{L\in\mathcal{G}_{s}\left[  I\right]  }\left(
\int_{\mathbb{R}^{3}}\left\vert T_{L,\mathbf{u}}f\left(  z\right)  \right\vert
\zeta_{\lambda}\left(  z-a\right)  dz\right)  ^{2}\right)  ^{\frac{1}{2}}.
\label{exclaim square}%
\end{equation}

Note that
\begin{equation}
\mathcal{S}_{\operatorname*{Fourier}}^{s}\mathsf{Q}_{I_{0}^{a}}^{s}f\left(
\xi\right)  =c\mathcal{S}_{\operatorname*{Fourier}}^{s}T_{I_{0}^{a}}f\left(
\xi\right)  , \label{Q and T}%
\end{equation}
where $T_{L}f\left(  \xi\right)  =\widehat{\Phi_{\ast}\left[  \bigtriangleup
_{L}f\left(  y\right)  dy\right]  }\left(  \xi\right)  $ by (\ref{pseudo}),
which has already been implicitly used in (\ref{w/o}). We note in passing that%
\begin{align*}
T_{I,\mathbf{u}}f\left(  \xi\right)   &  =e^{-i\xi\cdot\Phi\left(
c_{I}\right)  }\int e^{i\xi\cdot\Phi\left(  y\right)  }\mathsf{Q}%
_{I,\mathbf{u}}^{s}f\left(  y\right)  dy=e^{-i\xi\cdot\Phi\left(
c_{I}\right)  }\int e^{i\xi\cdot\Phi\left(  y\right)  }\sum_{L\in
\mathcal{G}_{s}\left[  I\right]  }\bigtriangleup_{L,\mathbf{u}}f\left(
y\right)  dy\\
&  =\sum_{L\in\mathcal{G}_{s}\left[  I\right]  }e^{-i\xi\cdot\left[
\Phi\left(  c_{I}\right)  -\Phi\left(  c_{L}\right)  \right]  }\left\{  \int
e^{i\xi\cdot\Phi\left(  y\right)  }e^{-i\xi\cdot\Phi\left(  c_{L}\right)
}\bigtriangleup_{L,\mathbf{u}}f\left(  y\right)  dy\right\} \\
&  =\sum_{L\in\mathcal{G}_{s}\left[  I\right]  }e^{-i\xi\cdot\left[
\Phi\left(  c_{I}\right)  -\Phi\left(  c_{L}\right)  \right]  }T_{L,\mathbf{u}%
}f\left(  \xi\right)  .
\end{align*}

For $L\in\mathcal{G}_{s}\left[  U\right]  $, and with this same $\lambda$
fixed for the moment, we define%
\begin{align*}
&  w_{L,\mathbf{u}}^{a}\left(  f\right)  \equiv\int_{\mathbb{R}^{3}}\left\vert
T_{L,\mathbf{u}}f\left(  z\right)  \right\vert \zeta_{\lambda}\left(
z-a\right)  dz=\int_{\mathbb{R}^{3}}\left\vert T_{L,\mathbf{u}}f\left(
z\right)  \right\vert \zeta\left(  \frac{z-a}{2^{\lambda}}\right)  \frac
{dz}{2^{3\lambda}}\\
&  =\int_{\mathbb{R}^{3}}\left\vert \widehat{f_{L,\mathbf{u}}^{\Phi}}\left(
z\right)  \right\vert \zeta\left(  \frac{z-a}{2^{\lambda}}\right)  \frac
{dz}{2^{3\lambda}}\approx\frac{1}{\left\vert B\left(  a,2^{\lambda}\right)
\right\vert }\int_{B\left(  a,2^{\lambda}\right)  }\left\vert \widehat
{f_{L,\mathbf{u}}^{\Phi}}\left(  z\right)  \right\vert \ ,
\end{align*}
and refer to $w_{L,\mathbf{u}}^{a}\left(  f\right)  $ as the `weight' of
$\widehat{f_{L,\mathbf{u}}^{\Phi}}$ relative to the ball $B\left(
a,2^{\lambda}\right)  $, which is comparable to the average of $\left\vert
\widehat{f_{L,\mathbf{u}}^{\Phi}}\left(  z\right)  \right\vert $ over the ball
$B\left(  a,2^{\lambda}\right)  $.

In analogy with this, for $I\in\mathcal{G}_{\lambda}\left[  S\right]  $,
define the right hand side of (\ref{exclaim square}) to be
\begin{align*}
\omega_{I,\mathbf{u}}^{a,s}\left(  f\right)   &  \equiv\left(  \sum
_{L\in\mathcal{G}_{s}\left[  I\right]  }\left(  \int_{\mathbb{R}^{3}%
}\left\vert T_{L,\mathbf{u}}f\left(  z\right)  \right\vert \zeta_{\lambda
}\left(  z-a\right)  dz\right)  ^{2}\right)  ^{\frac{1}{2}}=\left(  \sum
_{L\in\mathcal{G}_{s}\left[  I\right]  }w_{L,\mathbf{u}}^{a}\left(  f\right)
^{2}\right)  ^{\frac{1}{2}}\\
&  \approx\left(  \sum_{L\in\mathcal{G}_{s}\left[  I\right]  }\left(  \frac
{1}{\left\vert B\left(  a,2^{\lambda}\right)  \right\vert }\int_{B\left(
a,2^{\lambda}\right)  }\left\vert \widehat{f_{L,\mathbf{u}}^{\Phi}}\left(
z\right)  \right\vert \right)  ^{2}\right)  ^{\frac{1}{2}}\ ,
\end{align*}
and refer to $\omega_{I,\mathbf{u}}^{a,s}\left(  f\right)  $ as the `square
function weight' of $\widehat{f_{I,\mathbf{u}}^{\Phi}}$ relative to the ball
$B\left(  a,2^{\lambda}\right)  $ at scale $s$, which is the $\ell^{2}$ norm
of the averages $w_{L,\mathbf{u}}^{a}\left(  f\right)  $. Finally, we have
\[
\omega_{I,\mathbf{u}}^{a,s}\left(  f\right)  \lesssim\left\Vert \widehat
{f_{I,\mathbf{u}}^{\Phi}}\right\Vert _{L^{\infty}}\lesssim\left\Vert
Q_{I,\mathbf{u}}^{s}f\right\Vert _{L^{1}}\lesssim\left\vert I\right\vert
\left\Vert f\right\Vert _{L^{\infty}}=2^{-2\lambda}.
\]

Summarizing, we have%
\begin{equation}
\left\vert \mathcal{S}_{\operatorname*{Fourier}}^{s}T_{I,\mathbf{u}}f\left(
\xi\right)  \right\vert \leq\int_{\mathbb{R}^{3}}\left\vert \mathcal{S}%
_{\operatorname*{Fourier}}^{s}T_{I,\mathbf{u}}f\left(  z\right)  \right\vert
\zeta_{\lambda}\left(  z-a\right)  dz=\omega_{I,\mathbf{u}}^{a,s}\left(
f\right)  \ ,\ \ \ \ \ \text{for }\xi\in B\left(  a,2^{\lambda}\right)  .
\label{pre exclaim'}%
\end{equation}
and
\begin{equation}
\int_{\mathbb{R}^{3}}\left\vert \mathcal{S}_{\operatorname*{Fourier}}%
^{s}T_{I,\mathbf{u}}f\left(  z\right)  \right\vert \zeta_{\lambda}\left(
z-\xi\right)  dz\approx\omega_{I,\mathbf{u}}^{a,s}\left(  f\right)  \text{
},\ \ \ \ \ \text{for }\xi\in B\left(  a,2^{\lambda}\right)  ,
\label{exclaim'}%
\end{equation}
since for $\xi\in B\left(  a,2^{\lambda}\right)  $ we have $\zeta_{\lambda
}\left(  \xi-z\right)  \approx\zeta_{\lambda}\left(  a-z\right)
=\zeta_{\lambda}\left(  z-a\right)  $ by (\ref{pre note that}) and
(\ref{note that}).

Now set%
\[
\omega_{\ast,\mathbf{u}}^{a,s}\left(  f\right)  \equiv\max_{I\in
\mathcal{G}_{\lambda}\left[  S\right]  }\omega_{I,\mathbf{u}}^{a,s}\left(
f\right)  =\max_{I\in\mathcal{G}_{\lambda}\left[  S\right]  }\int
_{\mathbb{R}^{3}}\left\vert \mathcal{S}_{\operatorname*{Fourier}}%
^{s}T_{I,\mathbf{u}}f\left(  z\right)  \right\vert \zeta_{\lambda}\left(
z-a\right)  dz,
\]
and fix $I_{\ast}^{a}$ such that
\[
\omega_{I_{\ast}^{a},\mathbf{u}}^{a,s}\left(  f\right)  =\omega_{\ast
,\mathbf{u}}^{a,s}\left(  f\right)  .
\]
We note in passing that $\omega_{\ast,\mathbf{u}}^{a,s}\left(  f\right)
\rightarrow0$ as $\left\vert a\right\vert \rightarrow\infty$ by the
Riemann-Lebesgue Lemma.

For $\lambda>1$ chosen sufficiently large, we will now estimate the
contributions to the norm $\left\Vert \mathcal{S}_{\operatorname*{Fourier}%
}^{s,\mathbf{u}}f\right\Vert _{L^{q}\left(  \mathbb{R}^{3}\right)
}=\left\Vert \mathcal{S}_{\operatorname*{Fourier}}^{s}T_{\mathbf{u}%
}f\right\Vert _{L^{q}\left(  \mathbb{R}^{3}\right)  }$ in three exhaustive
cases in turn. The first case will yield the growth factor $2^{\frac
{\varepsilon}{3}s}$, while the next two cases will be absorbed.
\end{proof}

\begin{notation}
At this point in the proof, in order to reduce clutter of notation, we will
drop dependence of the Fourier square function weights $\omega_{I,\mathbf{u}%
}^{a,s}\left(  f\right)  $, functions $f_{I,\mathbf{u}}^{\Phi}$ and operators
$T_{I,\mathbf{u}}$, on the sequence $\mathbf{u}=\left\{  u_{I}\right\}
_{I\in\mathcal{G}_{\lambda}\left[  S\right]  }$ of translation vectors
$u_{I}\in\mathbb{R}^{3}$. This is reasonable since the vectors $u_{I}$ play
only the role of a parameter, due in large part to the intertwining identity
(\ref{comm}), which allows us to simply replace the pseudoprojections
$\mathsf{Q}_{I}^{s}f$ with the modulated pseudoprojections $\mathsf{Q}%
_{I,\mathbf{u}}^{s}f$ throughout the arguments. In addition, we further drop
the dependence of the weights $\omega_{I}^{a,s}\left(  f\right)  $ on $f$, and
simply write $\omega_{I}^{a,s}$. Of course the vectors $u_{I}$ and the
function $f$ will change under the parabolic rescalings in \textbf{Case}
\textbf{2} and \textbf{Case} \textbf{3} below, but this is harmless in light
of the definition of $A_{s}^{\left(  q\right)  }$, and the parabolic
invariance of the measures $\mathsf{M}_{\mathbf{u}}^{s}\Phi_{\ast}%
\mathsf{Q}_{I}^{s}f=\Phi_{\ast}\mathsf{Q}_{I,\mathbf{u}}^{s}f\in\mathcal{F}$.
The reader should keep this in mind in the following arguments.
\end{notation}

\subsubsection{Case 1: Separated interaction\label{Sub Case 1}}

\begin{proof}
[Proof continued]In \textbf{Case 1} we assume the following property holds for
the point $a$. There exists a triple of squares $I_{0}^{a},J_{0}^{a},K_{0}%
^{a}\in\mathcal{G}_{\lambda}\left[  U\right]  $ depending on $a$ such that%
\begin{align*}
\omega_{I_{0}^{a}}^{a,s},\omega_{J_{0}^{a}}^{a,s},\omega_{K_{0}^{a}}^{a,s}  &
>2^{-2\lambda}\omega_{\ast}^{a,s}\ ,\ \ \ \ \ \text{and }\left\vert
\mathbf{c}_{I_{0}^{a}}-\mathbf{c}_{J_{0}^{a}}\right\vert ,\left\vert
\mathbf{c}_{J_{0}^{a}}-\mathbf{c}_{K_{0}^{a}}\right\vert ,\left\vert
\mathbf{c}_{K_{0}^{a}}-\mathbf{c}_{I_{0}^{a}}\right\vert >2^{10}2^{-\lambda
}\ ,\\
&  \fbox{$%
\begin{array}
[c]{ccccc}%
\mathbf{c}_{I_{0}^{a,s}} &  &  &  & \\
& \cdot &  &  & \\
&  & \mathbf{c}_{K_{0}^{a,s}} & \leftrightarrows & \mathbf{c}_{K_{0}^{a,s}}\\
&  &  & \cdot & \\
&  &  &  & \mathbf{c}_{J_{0}^{a,s}}%
\end{array}
$},
\end{align*}
i.e. there exists a `$2^{10}2^{-\lambda}$-separated' triple $I_{0}^{a,s}%
,J_{0}^{a,s},K_{0}^{a,s}$ of squares of side length $2^{-\lambda} $, such that
each of $I_{0}^{a,s}$, $J_{0}^{a,s}\ $and $K_{0}^{a,s}$ have near maximal
weight. In \textbf{Case 1} we will use the $\nu$-disjoint trilinear square
function estimate $\mathcal{A}_{\operatorname*{disj}\nu}%
^{\operatorname*{square}}\left(  \otimes_{3}L^{\infty}\rightarrow L^{\frac
{q}{3}};\varepsilon\right)  $ with $\nu=2^{10}2^{-\lambda}$. For $\xi\in
B\left(  a,2^{\lambda}\right)  $ we throw away the modulation $e^{-i\Phi
\left(  c_{I}\right)  \cdot\xi}$, and using (\ref{pre exclaim'}), we estimate
that for $\xi\in B\left(  a,2^{\lambda}\right)  $,%

\begin{align}
& \label{T weight}\\
&  \left\vert \mathcal{S}_{\operatorname*{Fourier}}^{s}f\left(  \xi\right)
\right\vert =\left(  \sum_{L\in\mathcal{G}_{s}\left[  S\right]  }\left\vert
T_{L}f\left(  \xi\right)  \right\vert ^{2}\right)  ^{\frac{1}{2}}=\left(
\sum_{I\in\mathcal{G}_{\lambda}\left[  S\right]  }\sum_{L\in\mathcal{G}%
_{s}\left[  I\right]  }\left\vert T_{L}f\left(  \xi\right)  \right\vert
^{2}\right)  ^{\frac{1}{2}}\lesssim\left(  \sum_{I\in\mathcal{G}_{\lambda
}\left[  S\right]  }\left(  \sum_{L\in\mathcal{G}_{s}\left[  I\right]
}\left(  w_{L}^{a}\right)  ^{2}\right)  \right)  ^{\frac{1}{2}}\nonumber\\
&  =\left(  \sum_{I\in\mathcal{G}_{\lambda}\left[  S\right]  }\left(
\omega_{I}^{a,s}\right)  ^{2}\right)  ^{\frac{1}{2}}\leq\left(  \sum
_{I\in\mathcal{G}_{\lambda}\left[  S\right]  }\left(  \omega_{\ast}%
^{a,s}\right)  ^{2}\right)  ^{\frac{1}{2}}<2^{\lambda}\omega_{\ast}%
^{a,s}<2^{3\lambda}\left(  \omega_{I_{0}^{a}}^{a,s}\omega_{J_{0}^{a}}%
^{a,s}\omega_{K_{0}^{a}}^{a,s}\right)  ^{\frac{1}{3}},\nonumber
\end{align}
since the\emph{\ fixed} triple $\left(  I_{0}^{a},J_{0}^{a},K_{0}^{a}\right)
$ satisfies the near maximal weight condition in \textbf{Case 1}:%
\[
\omega_{\ast}^{a,s}<\min\left\{  2^{2\lambda}\omega_{I_{0}^{a}}^{a,s}%
,2^{2\lambda}\omega_{J_{0}^{a}}^{a,s},2^{2\lambda}\omega_{K_{0}^{a}}%
^{a,s}\right\}  \leq2^{2\lambda}\left(  \omega_{I_{0}^{a}}^{a,s}\right)
^{\frac{1}{3}}\left(  \omega_{J_{0}^{a}}^{a,s}\right)  ^{\frac{1}{3}}\left(
\omega_{K_{0}^{a}}^{a,s}\right)  ^{\frac{1}{3}}.
\]

We will now use the shorthand notation $\mathcal{S}_{\operatorname*{Fourier}%
}^{s}=\mathcal{S}^{s}$ in large displays. Thus for $q>3$ and $\xi\in B\left(
a,2^{\lambda}\right)  $, we have from (\ref{T weight}) and (\ref{exclaim'}),
that%
\begin{align}
& \label{S inequ}\\
&  \left\vert \mathcal{S}_{\operatorname*{Fourier}}^{s}f\left(  \xi\right)
\right\vert ^{q}\leq2^{q3\lambda}\left(  \omega_{I_{0}^{a}}^{a,s}\omega
_{J_{0}^{a}}^{a,s}\omega_{K_{0}^{a}}^{a,s}\right)  ^{\frac{q}{3}}\nonumber\\
&  \approx2^{q3\lambda}\left(  \int_{\mathbb{R}^{3}}\left\vert \mathcal{S}%
^{s}T_{I_{0}^{a}}f\left(  z_{1}\right)  \right\vert \zeta_{\lambda}\left(
z_{1}-a\right)  dz_{1}\right)  ^{\frac{q}{3}}\left(  \int_{\mathbb{R}^{3}%
}\left\vert \mathcal{S}^{s}T_{J_{0}^{a}}f\left(  z_{2}\right)  \right\vert
\zeta_{\lambda}\left(  z_{2}-a\right)  dz_{2}\right)  ^{\frac{q}{3}%
}\nonumber\\
&
\ \ \ \ \ \ \ \ \ \ \ \ \ \ \ \ \ \ \ \ \ \ \ \ \ \ \ \ \ \ \ \ \ \ \ \ \ \ \ \ \ \times
\left(  \int_{\mathbb{R}^{3}}\left\vert \mathcal{S}^{s}T_{K_{0}^{a}}f\left(
z_{3}\right)  \right\vert \zeta_{\lambda}\left(  z_{3}-a\right)
dz_{3}\right)  ^{\frac{q}{3}}\nonumber\\
&  \approx2^{q3\lambda}\left(  \int_{\mathbb{R}^{3}}\int_{\mathbb{R}^{3}}%
\int_{\mathbb{R}^{3}}\left\vert \mathcal{S}^{s}T_{I_{0}^{a}}f\left(  \xi
-z_{1}\right)  \mathcal{S}^{s}T_{J_{0}^{a}}f\left(  \xi-z_{2}\right)
\mathcal{S}^{s}T_{K_{0}^{a}}f\left(  \xi-z_{3}\right)  \right\vert \zeta
_{s}\left(  z_{1}\right)  \zeta_{s}\left(  z_{2}\right)  \zeta_{s}\left(
z_{3}\right)  dz_{1}dz_{2}dz_{3}\right)  ^{\frac{q}{3}}\ .\nonumber
\end{align}

Recall that given $q>3$ we assume there is $\nu>0$ such that $\mathcal{A}%
_{\operatorname*{disj}\nu}^{\operatorname*{square}}\left(  \otimes
_{3}L^{\infty}\rightarrow L^{\frac{q}{3}};\varepsilon\right)  $ holds for all
$0<\varepsilon<1$. So we now fix $q>3$, and let $\nu=\nu_{q}$ be the
corresponding disjoint parameter. We will now take $\lambda>1$ sufficiently
large that
\begin{equation}
2^{10}2^{-\lambda}\leq\nu=\nu_{q}, \label{q control}%
\end{equation}
and use the $\nu$-disjoint trilinear square function bound
(\ref{single tri Four}), with
\[
f_{1}=f_{I_{0}^{a}}^{\Phi},\ \ \ f_{2}=f_{J_{0}^{a}}^{\Phi},\ \ \ \ f_{3}%
=f_{K_{0}^{a}}^{\Phi}.
\]
We also recall that
\[
\mathcal{S}_{\operatorname*{Fourier}}^{s}\mathsf{Q}_{I}^{s}f\left(
\xi\right)  \approx\omega_{I}^{a,s}\leq\omega_{\ast}^{a,s}\leq2^{2\lambda
}\omega_{I_{0}^{a}}^{a,s}\approx2^{2\lambda}\mathcal{S}%
_{\operatorname*{Fourier}}^{s}\mathsf{Q}_{I_{0}^{a}}^{s}f\left(  \xi\right)
,\ \ \ \ \ \text{for }\xi\in B\left(  a,2^{\lambda}\right)  ,
\]
and so also%
\[
\mathcal{S}_{\operatorname*{Fourier}}^{s}T_{I}f\left(  \xi\right)
\lesssim2^{2\lambda}\mathcal{S}_{\operatorname*{Fourier}}^{s}T_{I_{0}^{a}%
}f\left(  \xi\right)  ,\ \ \ \ \ \text{for }\xi\in B\left(  a,2^{\lambda
}\right)  .
\]
Thus we obtain from (\ref{single tri Four}) that%
\begin{align*}
&  \int_{\mathbb{R}^{3}}\left\vert \mathcal{S}_{\operatorname*{Fourier}}%
^{s}\mathsf{Q}_{I_{0}^{a}}^{s}f\left(  \xi\right)  \mathcal{S}%
_{\operatorname*{Fourier}}^{s}\mathsf{Q}_{J_{0}^{a}}^{s}f\left(  \xi\right)
\mathcal{S}_{\operatorname*{Fourier}}^{s}\mathsf{Q}_{K_{0}^{a}}^{s}f\left(
\xi\right)  \right\vert ^{\frac{q}{3}}d\xi\\
&  \leq C_{\nu}2^{\varepsilon s\frac{q}{3}}\left\Vert \mathsf{Q}_{I_{0}^{a}%
}^{s}f\right\Vert _{L^{\infty}}\left\Vert \mathsf{Q}_{J_{0}^{a}}%
^{s}f\right\Vert _{L^{\infty}}\left\Vert \mathsf{Q}_{K_{0}^{a}}^{s}%
f\right\Vert _{L^{\infty}}=C_{\nu}2^{\varepsilon s\frac{q}{3}},
\end{align*}
since the triple $I_{0}^{a},J_{0}^{a},K_{0}^{a}$ is $2^{10}2^{-\lambda}%
$-separated, and since $\left\vert \mathsf{Q}_{I}^{s}f\right\vert \leq1$.
Similarly if $\left\vert z_{k}\right\vert \lesssim2^{\lambda}$, we also have%
\begin{equation}
\int_{\mathbb{R}^{3}}\left\vert \mathcal{S}_{\operatorname*{Fourier}}%
^{s}\mathsf{Q}_{I_{0}^{a}}^{s}f\left(  \xi-z_{1}\right)  \mathcal{S}%
_{\operatorname*{Fourier}}^{s}\mathsf{Q}_{J_{0}^{a}}^{s}f\left(  \xi
-z_{2}\right)  \mathcal{S}_{\operatorname*{Fourier}}^{s}\mathsf{Q}_{K_{0}^{a}%
}^{s}f\left(  \xi-z_{3}\right)  \right\vert ^{\frac{q}{3}}d\xi\leq C_{\nu
}2^{\varepsilon s\frac{q}{3}}, \label{with z's}%
\end{equation}
since%
\begin{align*}
\mathcal{S}_{\operatorname*{Fourier}}^{s}\mathsf{Q}_{I_{0}^{a}}^{s}f\left(
\xi-z_{1}\right)   &  \approx\mathcal{S}_{\operatorname*{Fourier}}^{s}%
T_{I_{0}^{a}}f\left(  \xi-z_{1}\right)  =\left(  \sum_{L\in\mathcal{G}%
_{s}\left[  I_{0}^{a}\right]  }\left\vert \widehat{\Phi_{\ast}\bigtriangleup
_{L}f}\left(  \xi-z_{1}\right)  \right\vert ^{2}\right)  ^{\frac{1}{2}}\\
&  =\left(  \sum_{L\in\mathcal{G}_{s}\left[  I_{0}^{a}\right]  }\left\vert
\left[  e^{iz_{1}\cdot z}\left(  \Phi_{\ast}\bigtriangleup_{L}f\right)
\left(  z\right)  \right]  ^{\wedge}\left(  \xi\right)  \right\vert
^{2}\right)  ^{\frac{1}{2}}%
\end{align*}
and the exponential factor $e^{iz_{1}\cdot z}$ can be incorporated into the
modulation $\mathsf{M}_{\mathbf{u}}^{s}$.

Now consider those points $a\in\Gamma_{\lambda}$ for which \textbf{Case 1} is
in effect for the ball $B\left(  a,2^{\lambda}\right)  $ and denote by
$\Gamma_{\lambda}\left(  \text{\textbf{Case 1}}\right)  $ the union of all the
balls $B\left(  a,2^{\lambda}\right)  $ for which $a$ is in \textbf{Case 1}.

Summing over points $a\in\Gamma_{\lambda}$ such that \textbf{Case 1} is in
effect for the ball $B\left(  a,2^{\lambda}\right)  $, and using
(\ref{S inequ}), shows that the contribution $\left\Vert \mathcal{S}%
_{\operatorname*{Fourier}}^{s}f\right\Vert _{L^{q}\left(  \mathbf{1}%
_{\Gamma_{\lambda}\left(  \text{\textbf{Case 1}}\right)  }\mathbb{R}%
^{3}\right)  }$ to the norm $\left\Vert \mathcal{S}_{\operatorname*{Fourier}%
}^{s}f\right\Vert _{L^{q}\left(  \mathbb{R}^{3}\right)  }$ in \textbf{Case 1} satisfies,%

\begin{align*}
&  \ \ \ \ \ \ \ \ \ \ \ \ \ \ \ \left\Vert \mathcal{S}%
_{\operatorname*{Fourier}}^{s}f\right\Vert _{L^{q}\left(  \mathbf{1}%
_{\Gamma_{\lambda}\left(  \text{\textbf{Case 1}}\right)  }\mathbb{R}%
^{3}\right)  }^{q}\lesssim\sum_{a\in\Gamma_{\lambda}\left(  \text{\textbf{Case
1}}\right)  }\int_{B\left(  a,2^{\lambda}\right)  }\left\vert \mathcal{S}%
_{\operatorname*{Fourier}}^{s}f\left(  \xi\right)  \right\vert ^{q}d\xi\\
&  \lesssim\sum_{a\in\Gamma_{\lambda}\left(  \text{\textbf{Case 1}}\right)
}\int_{B\left(  a,2^{\lambda}\right)  }2^{q3\lambda}\int_{\mathbb{R}^{3}}%
\int_{\mathbb{R}^{3}}\int_{\mathbb{R}^{3}}\left\vert \mathcal{S}^{s}%
T_{I_{0}^{a}}f\left(  \xi-z_{1}\right)  \mathcal{S}^{s}T_{J_{0}^{a}}f\left(
\xi-z_{2}\right)  \mathcal{S}^{s}T_{K_{0}^{a}}f\left(  \xi-z_{3}\right)
\right\vert ^{\frac{q}{3}}\\
&  \ \ \ \ \ \ \ \ \ \ \ \ \ \ \ \ \ \ \ \ \times\left(  \zeta_{\lambda
}\left(  z_{1}\right)  \zeta_{\lambda}\left(  z_{2}\right)  \zeta_{\lambda
}\left(  z_{3}\right)  \right)  ^{\frac{q}{3}}dz_{1}dz_{2}dz_{3}d\xi\\
&  =2^{q3\lambda}\int_{\mathbb{R}^{9}}\left\{  \sum_{a\in\Gamma_{\lambda}}%
\int_{B\left(  a,2^{\lambda}\right)  }\left\vert \mathcal{S}^{s}T_{I_{0}^{a}%
}f\left(  \xi-z_{1}\right)  \mathcal{S}^{s}T_{J_{0}^{a}}f\left(  \xi
-z_{2}\right)  \mathcal{S}^{s}T_{K_{0}^{a}}f\left(  \xi-z_{3}\right)
\right\vert ^{\frac{q}{3}}d\xi\right. \\
&  \ \ \ \ \ \ \ \ \ \ \ \ \ \ \ \ \ \ \ \ \left.  \times\left(
\zeta_{\lambda}\left(  z_{1}\right)  \zeta_{\lambda}\left(  z_{2}\right)
\zeta_{\lambda}\left(  z_{3}\right)  \right)  ^{\frac{q}{3}}\right\}
\ dz_{1}dz_{2}dz_{3}\\
&  \leq2^{q3\lambda}\int_{\mathbb{R}^{9}}\left\{  \sum_{a\in\Gamma_{\lambda}%
}\int_{B\left(  a,2^{\lambda}\right)  }\sum_{I,J,K\in\mathcal{G}_{\lambda
}\left[  U\right]  }\left\vert \mathcal{S}^{s}T_{I}f\left(  \xi-z_{1}\right)
\mathcal{S}^{s}T_{J}f\left(  \xi-z_{2}\right)  \mathcal{S}^{s}T_{K}f\left(
\xi-z_{3}\right)  \right\vert ^{\frac{q}{3}}d\xi\right. \\
&  \ \ \ \ \ \ \ \ \ \ \ \ \ \ \ \ \ \ \ \ \left.  \times\left(
\zeta_{\lambda}\left(  z_{1}\right)  \zeta_{\lambda}\left(  z_{2}\right)
\zeta_{\lambda}\left(  z_{3}\right)  \right)  ^{\frac{q}{3}}\right\}
\ dz_{1}dz_{2}dz_{3}%
\end{align*}
which equals%
\begin{align*}
&  2^{q3\lambda}\int_{\mathbb{R}^{9}}\left\{  \sum_{I,J,K\in\mathcal{G}%
_{\lambda}\left[  U\right]  }\int_{\mathbb{R}^{3}}\sum_{a\in\Gamma_{\lambda}%
}\mathbf{1}_{B\left(  a,2^{\lambda}\right)  }\left\vert \mathcal{S}^{s}%
T_{I}f\left(  \xi-z_{1}\right)  \mathcal{S}^{s}T_{J}f\left(  \xi-z_{2}\right)
\mathcal{S}^{s}T_{K}f\left(  \xi-z_{3}\right)  \right\vert ^{\frac{q}{3}}%
d\xi\right. \\
&  \ \ \ \ \ \ \ \ \ \ \ \ \ \ \ \ \ \ \ \ \left.  \times\left(
\zeta_{\lambda}\left(  z_{1}\right)  \zeta_{\lambda}\left(  z_{2}\right)
\zeta_{\lambda}\left(  z_{3}\right)  \right)  ^{\frac{q}{3}}\right\}
\ dz_{1}dz_{2}dz_{3}\\
&  \lesssim2^{q3\lambda}\int_{\mathbb{R}^{9}}\left(  \int_{\mathbb{R}^{3}}%
\sum_{I,J,K\in\mathcal{G}_{\lambda}\left[  U\right]  }\left\vert
\mathcal{S}^{s}T_{I}f\left(  \xi-z_{1}\right)  \mathcal{S}^{s}T_{J}f\left(
\xi-z_{2}\right)  \mathcal{S}^{s}T_{K}f\left(  \xi-z_{3}\right)  \right\vert
^{\frac{q}{3}}d\xi\right) \\
&  \ \ \ \ \ \ \ \ \ \ \ \ \ \ \ \ \ \ \ \ \times\left(  \zeta_{\lambda
}\left(  z_{1}\right)  \zeta_{\lambda}\left(  z_{2}\right)  \zeta_{\lambda
}\left(  z_{3}\right)  \right)  ^{\frac{q}{3}}\ dz_{1}dz_{2}dz_{3}\\
&  \lesssim2^{q3\lambda}C_{\nu}2^{6\lambda}2^{\varepsilon s\frac{q}{3}}%
\int_{\mathbb{R}^{9}}\left(  \zeta_{\lambda}\left(  z_{1}\right)
\zeta_{\lambda}\left(  z_{2}\right)  \zeta_{\lambda}\left(  z_{3}\right)
\right)  ^{\frac{q}{3}}\ dz_{1}dz_{2}dz_{3}\\
&  \approx2^{q3\lambda}C_{\nu}2^{6\lambda}2^{\varepsilon s\frac{q}{3}}\frac
{1}{2^{3q\lambda}}2^{9\lambda}=C_{\nu}2^{15\lambda}2^{\varepsilon s\frac{q}%
{3}},
\end{align*}
since $\zeta_{\lambda}\left(  z\right)  \approx\frac{1}{2^{3\lambda}}$ on
$B\left(  0,2^{\lambda}\right)  $, and where in line 5 we have used
(\ref{with z's}). In other words we have $\left\Vert \mathbf{1}_{\Gamma
_{\lambda}\left(  \text{\textbf{Case 1}}\right)  }\mathcal{S}^{s}Tf\right\Vert
_{L^{q}\left(  \mathbb{R}^{3}\right)  }\lesssim C_{\nu}2^{\left(  3+\frac
{15}{q}\right)  \lambda}2^{\frac{\varepsilon}{3}s}$, where $\mathbf{1}%
_{\Gamma_{\lambda}\left(  \text{\textbf{Case 1}}\right)  }$ is the indicator
function of the union of those balls $B\left(  a,2^{\lambda}\right)  $ for
which \textbf{Case 1} holds.
\end{proof}

\subsubsection{Case 2: Clustered interaction\label{Sub Case 2}}

\begin{proof}
[Proof continued]In \textbf{Case 2} we assume the following property. If
$\left\vert c_{I}-c_{I_{\ast}^{a}}\right\vert >2^{10}2^{-\lambda}$, then
$\omega_{I}^{a,s}\leq2^{-2\lambda}w_{\ast}^{a,s}$. In other words, if $I$ is
sufficiently far from $I_{\ast}^{a,s}$, then $w_{I}^{a,s}$ is much smaller
than $w_{\ast}^{a,s}$, i.e.%
\begin{align*}
\operatorname*{dist}\left(  I,I_{\ast}^{a,s}\right)   &  >2^{10}2^{-\lambda
}\Longrightarrow\int_{\mathbb{R}^{3}}\left\vert \mathcal{S}%
_{\operatorname*{Fourier}}^{s}T_{I}f\left(  z\right)  \right\vert
\zeta_{\lambda}\left(  z-a\right)  dz\leq2^{-2\lambda}\int_{\mathbb{R}^{3}%
}\left\vert \mathcal{S}_{\operatorname*{Fourier}}^{s}T_{I_{\ast}^{a}}f\left(
z\right)  \right\vert \zeta_{\lambda}\left(  z-a\right)  dz,\\
&  \fbox{$%
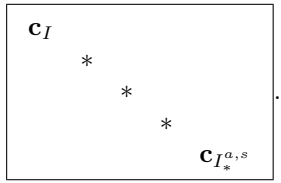
\begin{array}
[c]{ccccc}%
\mathbf{c}_{I} &  &  &  & \\
& \ast &  &  & \\
&  & \ast &  & \\
&  &  & \ast & \\
&  &  &  & \mathbf{c}_{I_{\ast}^{a,s}}%
\end{array}
$}.
\end{align*}

Here we will use parabolic rescaling as in \cite{TaVaVe}. Recall that
$\phi\left(  \xi,y\right)  =\xi\cdot\Phi\left(  y\right)  $. Let
$\lambda^{\prime}\equiv\frac{1}{2}\lambda-10$, and let $\xi\in B\left(
a,2^{\lambda}\right)  $ for some $a\in\Gamma_{\lambda}$. Using that
$\left\vert c_{I}-c_{I_{\ast}^{a,s}}\right\vert >2^{10}2^{-\lambda}$ implies
$\omega_{I}^{a,s}\leq2^{-2\lambda}\omega_{\ast}^{a,s}$ in this case, we have
with $\left\vert x\right\vert _{\infty}\equiv\max\left\{  \left\vert
x_{1}\right\vert ,\left\vert x_{2}\right\vert \right\}  $ that
\begin{align*}
\left\vert \mathcal{S}_{\operatorname*{Fourier}}^{s}f\left(  \xi\right)
\right\vert  &  =\left(  \sum_{I\in\mathcal{G}_{\lambda}\left[  U\right]
}\sum_{L\in\mathcal{G}_{s}\left[  I\right]  }\left\vert \int_{U}%
e^{i\phi\left(  \xi,y\right)  }\bigtriangleup_{L}f\left(  y\right)
dy\right\vert ^{2}\right)  ^{\frac{1}{2}}=\left(  \sum_{I\in\mathcal{G}%
_{\lambda}\left[  U\right]  }\mathcal{S}_{\operatorname*{Fourier}}%
^{s}\mathsf{Q}_{I}^{s}f\left(  \xi\right)  ^{2}\right)  ^{\frac{1}{2}}\\
&  \leq\left(  \sum_{I\in\mathcal{G}_{\lambda}\left[  U\right]  :\ \left\vert
c_{I}-c_{I_{\ast}^{a}}\right\vert _{\infty}\leq2^{-\lambda^{\prime}}}%
\sum_{L\in\mathcal{G}_{s}\left[  I\right]  }\left\vert \int_{U}e^{i\phi\left(
\xi,y\right)  }\bigtriangleup_{L}f\left(  y\right)  dy\right\vert ^{2}\right)
^{\frac{1}{2}}+\left(  \sum_{I\in\mathcal{G}_{\lambda}\left[  U\right]
:\ \left\vert c_{I}-c_{I_{\ast}^{a}}\right\vert _{\infty}>2^{-\lambda^{\prime
}}}\mathcal{S}_{\operatorname*{Fourier}}^{s}\mathsf{Q}_{I}^{s}f\left(
\xi\right)  ^{2}\right)  ^{\frac{1}{2}}\\
&  \leq\left(  \sum_{L\in\mathcal{G}_{s}\left[  \left\{  x:\left\vert
x-c_{I_{\ast}^{a}}\right\vert _{\infty}\leq2^{-\lambda^{\prime}}\right\}
\right]  }\left\vert \int_{U}e^{i\phi\left(  \xi,y\right)  }\bigtriangleup
_{L}f\left(  y\right)  dy\right\vert ^{2}\right)  ^{\frac{1}{2}}+\left(
\sum_{I\in\mathcal{G}_{\lambda}\left[  U\right]  :\ \left\vert c_{I}%
-c_{I_{\ast}^{a}}\right\vert _{\infty}>2^{-\lambda^{\prime}}}\left(
\omega_{I}^{a.s}\right)  ^{2}\right)  ^{\frac{1}{2}},
\end{align*}
which is at most%
\begin{align*}
&  10\max_{K\in\mathcal{G}_{\lambda^{\prime}}\left[  U\right]  }\left(
\sum_{L\in\mathcal{G}_{s}\left[  K\right]  }\left\vert \int_{U}e^{i\phi\left(
\xi,y\right)  }\bigtriangleup_{L}f\left(  y\right)  dy\right\vert ^{2}\right)
^{\frac{1}{2}}+\left(  \sum_{I\in\mathcal{G}_{\lambda}\left[  U\right]
:\ \left\vert c_{I}-c_{I_{\ast}^{a}}\right\vert _{\infty}>2^{-\lambda^{\prime
}}}\left(  \omega_{I}^{a.s}\right)  ^{2}\right)  ^{\frac{1}{2}}\\
&  \leq10\max_{K\in\mathcal{G}_{\lambda^{\prime}}\left[  U\right]
}\mathcal{S}^{s}T_{K}f\left(  \xi\right)  +\left(  \#\mathcal{G}_{\lambda
}\left[  U\right]  \right)  ^{\frac{1}{2}}2^{-2\lambda}\omega_{\ast}^{a,s}%
\leq10\max_{K\in\mathcal{G}_{\frac{1}{2}\lambda-10}\left[  U\right]
}\mathcal{S}^{s}T_{K}f\left(  \xi\right)  +2^{-\lambda}\omega_{I_{\ast}^{a}%
}^{a,s}\ .
\end{align*}

Now%
\begin{align*}
\omega_{I_{\ast}^{a}}^{a,s}  &  =\int_{\mathbb{R}^{3}}\left\vert
\mathcal{S}_{\operatorname*{Fourier}}^{s}T_{I_{\ast}^{a}}f\left(  \xi\right)
\right\vert \zeta_{\lambda}^{a}\left(  \xi\right)  d\xi\leq\left(
\int_{\mathbb{R}^{3}}\left\vert \mathcal{S}_{\operatorname*{Fourier}}%
^{s}T_{I_{\ast}^{a}}f\left(  \xi\right)  \right\vert ^{q}\zeta_{\lambda}%
^{a}\left(  \xi\right)  d\xi\right)  ^{\frac{1}{q}}\left(  \int_{\mathbb{R}%
^{3}}\zeta_{\lambda}^{a}\left(  \xi\right)  d\xi\right)  ^{\frac{1}{q^{\prime
}}}\\
&  \leq\left(  \int_{\mathbb{R}^{3}}\left\vert \mathcal{S}%
_{\operatorname*{Fourier}}^{s}T_{I_{\ast}^{a}}f\left(  \xi\right)  \right\vert
^{q}\zeta_{\lambda}\left(  \xi-a\right)  d\xi\right)  ^{\frac{1}{q}}=\left(
\int_{\mathbb{R}^{3}}\left\vert \mathcal{S}_{\operatorname*{Fourier}}%
^{s}T_{I_{\ast}^{a}}f\left(  z\right)  \right\vert ^{q}\zeta_{\lambda}\left(
z-a\right)  dz\right)  ^{\frac{1}{q}},
\end{align*}
and so for $\xi\in B\left(  a,2^{\lambda}\right)  $,
\begin{align*}
\left\vert \mathcal{S}_{\operatorname*{Fourier}}^{s}f\left(  \xi\right)
\right\vert ^{q}  &  \leq C\sum_{K\in\mathcal{G}_{\lambda^{\prime}}\left[
U\right]  }\left\vert \mathcal{S}_{\operatorname*{Fourier}}^{s}T_{K}f\left(
\xi\right)  \right\vert ^{q}+C2^{-\lambda q}\int_{\mathbb{R}^{3}}\left\vert
\mathcal{S}_{\operatorname*{Fourier}}^{s}T_{I_{\ast}^{a}}f\left(  z\right)
\right\vert ^{q}\zeta_{\lambda}\left(  z-a\right)  dz\\
&  \leq C\sum_{K\in\mathcal{G}_{\lambda^{\prime}}\left[  U\right]  }\left\vert
\mathcal{S}_{\operatorname*{Fourier}}^{s}T_{K}f\left(  \xi\right)  \right\vert
^{q}+C2^{-\lambda q}\sum_{I\in\mathcal{G}_{\lambda}\left[  U\right]  }%
\int\left\vert \mathcal{S}_{\operatorname*{Fourier}}^{s}T_{I}f\left(
z\right)  \right\vert ^{q}\zeta_{\lambda}\left(  z-a\right)  dz,
\end{align*}
where we have added in all $K\in\mathcal{G}_{\lambda^{\prime}}\left[
U\right]  $ rather than just $K_{\ast}^{a}$, and all $I\in\mathcal{G}%
_{\lambda}\left[  U\right]  $ rather than just $I_{\ast}^{a}$.

Summing over $a\in\Gamma_{\lambda}$, we see that the corresponding
contribution in \textbf{Case 2} is at most%
\begin{align}
& \label{contrib Case 2}\\
\left\Vert \mathbf{1}_{\Gamma_{\lambda}\left(  \text{\textbf{Case 2}}\right)
}\mathcal{S}_{\operatorname*{Fourier}}^{s}f\right\Vert _{L^{q}\left(
\mathbb{R}^{3}\right)  }^{q}  &  \equiv\sum_{a\in\Gamma_{\lambda}}\left\{
C\sum_{K\in\mathcal{G}_{\lambda^{\prime}}\left[  U\right]  }\int_{B\left(
a,2^{s}\right)  }\left\vert \mathcal{S}^{s}T_{K}f\left(  \xi\right)
\right\vert ^{q}d\xi+C2^{-\lambda q}\sum_{I\in\mathcal{G}_{\lambda}\left[
U\right]  }\int_{\mathbb{R}^{3}}\left\vert \mathcal{S}^{s}T_{I}f\left(
z\right)  \right\vert ^{q}\zeta_{s}^{a}\left(  z\right)  dz\right\}
\nonumber\\
&  \lesssim C\sum_{K\in\mathcal{G}_{\lambda^{\prime}}\left[  U\right]  }%
\int_{\mathbb{R}^{3}}\left\vert T_{K}f\left(  \xi\right)  \right\vert ^{q}%
d\xi+C2^{-3\lambda}2^{\lambda q}\sum_{I\in\mathcal{G}_{\lambda}\left[
U\right]  }\int_{\mathbb{R}^{3}}\left\vert \mathcal{S}^{s}T_{I}f\left(
\xi\right)  \right\vert ^{q}d\xi,\nonumber
\end{align}
since $\sum_{a\in\Gamma_{\lambda}}\zeta_{\lambda}\left(  z-a\right)
\lesssim2^{-3\lambda}+\operatorname*{rapid}\operatorname*{decay}$.

At this point we follow \cite{BoGu} in using parabolic rescaling as introduced
in Tao, Vargas and Vega \cite{TaVaVe} on the integral
\[
\operatorname*{Int}_{\rho}\left(  \xi\right)  \equiv\int_{\left\vert
y-\overline{y}\right\vert <\rho}e^{i\phi\left(  \xi,y\right)  }f\left(
y\right)  dy=\int_{\left\vert y-\overline{y}\right\vert <\rho}e^{i\left[
\xi_{1}y_{1}+\xi_{2}y_{2}+\xi_{2}\left(  y_{1}^{2}+y_{2}^{2}\right)  \right]
}f\left(  y\right)  dy,\ \ \ \ \ \text{for }0<\rho<1,
\]
where we are temporarily setting $u_{I}=1$ for convenience, to obtain%
\begin{align*}
&  \left\vert \operatorname*{Int}_{\rho}\left(  \xi\right)  \right\vert
\overset{y=\overline{y}+y^{\prime}}{=}\left\vert \int_{\left\vert y^{\prime
}\right\vert <\rho}e^{i\left[  \xi_{1}\left(  \overline{y}_{1}+y_{1}^{\prime
}\right)  +\xi_{2}\left(  \overline{y}_{2}+y_{2}^{\prime}\right)  +\xi
_{3}\left(  \left(  \overline{y}_{1}+y_{1}^{\prime}\right)  ^{2}+\left(
\overline{y}_{2}+y_{2}^{\prime}\right)  ^{2}\right)  \right]  }f\left(
\overline{y}+y^{\prime}\right)  dy^{\prime}\right\vert \\
&  =\left\vert \int_{\left\vert y^{\prime}\right\vert <\rho}e^{i\left[
\left(  \xi_{1}+2\overline{y_{1}}\xi_{3}\right)  y_{1}^{\prime}+\left(
\xi_{2}+2\overline{y_{2}}\xi_{3}\right)  y_{2}^{\prime}+\xi_{3}\left\vert
y^{\prime}\right\vert ^{2}\right]  }f\left(  \overline{y}+y^{\prime}\right)
dy^{\prime}\right\vert ,
\end{align*}
and conclude that%
\begin{align*}
&  \ \ \ \ \ \ \ \ \ \ \ \ \ \ \ \ \ \ \ \ \ \ \ \ \ \ \ \ \ \ \left\Vert
\operatorname*{Int}_{\rho}\right\Vert _{L^{q}\left(  \mathbb{R}^{3}\right)
}=\left(  \int_{\mathbb{R}^{3}}\left\vert \operatorname*{Int}_{\rho}\left(
\xi\right)  \right\vert ^{q}d\xi\right)  ^{\frac{1}{q}}\\
&  =\left(  \int_{\mathbb{R}^{3}}\left\vert \int_{\left\vert y^{\prime
}\right\vert <\rho}e^{i\left[  \left(  \xi_{1}+2\overline{y_{1}}\xi
_{3}\right)  y_{1}^{\prime}+\left(  \xi_{2}+2\overline{y_{2}}\xi_{3}\right)
y_{2}^{\prime}+\xi_{3}\left\vert y^{\prime}\right\vert ^{2}\right]  }f\left(
\overline{y}+y^{\prime}\right)  dy^{\prime}\right\vert ^{q}d\xi\right)
^{\frac{1}{q}}\\
&  =\left(  \int_{\mathbb{R}^{3}}\left\vert \int_{\left\vert y^{\prime
}\right\vert <\rho}e^{i\left[  \left(  \rho\xi_{1}+2\overline{\frac{y_{1}%
}{\rho}}\rho^{2}\xi_{3}\right)  \frac{y_{1}^{\prime}}{\rho}+\left(  \rho
\xi_{2}+2\overline{\frac{y_{2}}{\rho}}\rho^{2}\xi_{3}\right)  y_{2}^{\prime
}+\rho^{2}\xi_{3}\left\vert \frac{y^{\prime}}{\rho}\right\vert ^{2}\right]
}f\left(  \overline{y}+y^{\prime}\right)  \rho^{2}d\left(  \frac{y^{\prime}%
}{\rho}\right)  \right\vert ^{q}\frac{d\left(  \rho\xi^{\prime}\right)
d\left(  \rho^{2}\xi_{3}\right)  }{\rho^{4}}\right)  ^{\frac{1}{q}}\\
&  =\rho^{2}\rho^{-\frac{4}{q}}\left(  \int_{\mathbb{R}^{3}}\left\vert
\int_{\left\vert y^{\prime}\right\vert <1}e^{i\left[  \left(  \xi
_{1}+2\overline{y_{1}}\xi_{3}\right)  y_{1}^{\prime}+\left(  \xi
_{2}+2\overline{y_{2}}\xi_{3}\right)  y_{2}^{\prime}+\xi_{3}\left\vert
y^{\prime}\right\vert ^{2}\right]  }f\left(  \rho\left(  \overline
{y}+y^{\prime}\right)  \right)  dy^{\prime}\right\vert ^{q}d\xi^{\prime}%
d\xi_{3}\right)  ^{\frac{1}{q}}.
\end{align*}
If $\rho$ is a\ dyadic number $2^{n}$ for $n\in\mathbb{Z}$, and if $f$ is
replaced by a projection $\mathsf{Q}_{s}f=\sum_{I\in\mathcal{G}_{\lambda
}\left[  U\right]  }\left\langle f,\varphi_{I}\right\rangle \varphi_{I}$
above, then we conclude that%
\begin{equation}
\left\Vert \operatorname*{Int}_{\rho}\right\Vert _{L^{q}\left(  \mathbb{R}%
^{3}\right)  }\leq C\rho^{2}\rho^{-\frac{4}{q}}\sup_{\mathcal{G}%
\in\operatorname*{Grid}}\sup_{\lambda\in\mathbb{N}}\sup_{\mathbf{u}%
\in\mathcal{V}}\sup_{\left\Vert f\right\Vert _{L^{\infty}\left(
\sigma\right)  }\leq1}\left(  \int_{\mathbb{R}^{3}}\left\vert \sum
_{I\in\mathcal{G}_{\lambda}\left[  U\right]  }T_{I,\mathbf{u}}f\left(
\xi\right)  \right\vert ^{q}d\xi\right)  ^{\frac{1}{q}}. \label{Int est}%
\end{equation}
One can also apply this same change of variable to a Fourier square function
$\mathcal{S}_{\operatorname*{Fourier}}^{\lambda}f$ with $2^{-\lambda}<\rho$ to
obtain the following Fourier square function analogue of (\ref{Int est}),%
\begin{equation}
\left\Vert \mathcal{S}_{\operatorname*{Fourier}}^{s}f\right\Vert
_{L^{q}\left(  \mathbb{R}^{3}\right)  }\leq C\rho^{2}\rho^{-\frac{4}{q}}%
A_{s}^{\left(  q\right)  }\ , \label{sq anal}%
\end{equation}
where we recall that $A_{s}^{\left(  q\right)  }$ is defined in (\ref{Q_R}),
i.e.%
\[
A_{s}^{\left(  q\right)  }\equiv\sup_{\mathcal{G}\in\operatorname*{Grid}}%
\sup_{\mathbf{u}\in\mathcal{V}}\sup_{\left\Vert f\right\Vert _{L^{\infty
}\left(  U\right)  }\leq1}\left(  \int_{\mathbb{R}^{3}}\left\vert
\mathcal{S}_{\operatorname*{Fourier}}^{s,\mathbf{u}}f\left(  \xi\right)
\right\vert ^{q}d\xi\right)  ^{\frac{1}{q}},
\]
(where we have reinstated the sequence $\mathbf{u}$ in the definition of
$\operatorname*{Int}_{\rho}$, and in the notation for the definition of
$A_{s}^{\left(  q\right)  }$), since%
\begin{align*}
&  \text{the }L^{\infty}\text{ norm of }f\text{ is unchanged by dilation,}\\
&  \text{the paraboloid is invariant under parabolic rescaling,}\\
&  \text{the parabolic dilate }\delta_{\rho}\text{ of the measure }\Phi_{\ast
}\varphi_{I}\text{ is }\Phi_{\ast}\varphi_{\delta_{\rho}I}\text{,}\\
&  \text{and the modulations }\mathsf{M}_{\mathbf{u}}^{s}\text{ are invariant
under (any) linear rescaling.}%
\end{align*}
\newline Note that this requires that the parabolic dilation of size $\frac
{1}{\rho}$ takes the support of $\Phi_{\ast}f$ into $\Phi\left(  U\right)  $.
This can always be arranged even for parabolic dilations relative to other
points on the paraboloid (such as points $\Phi\left(  x\right)  $ with $x$
near $\partial U$), by first applying a parabolic dilation of size $\frac
{1}{2}$ (i.e. $\rho=2$) relative to the \emph{origin} to $\Phi_{\ast}f$ (which
is harmless), so that its support is contained in $\Phi\left(  \frac{1}%
{2}U\right)  $, and then applying a parabolic dilation relative to a point $P$
in the support of $f$, which will keep the support of the dilated measure in
$\Phi\left(  U\right)  $.

Note also that the factor $\rho^{2}$ arises from $\left\vert y^{\prime
}\right\vert <\rho$, and that the factor $\rho^{-\frac{4}{q}}A_{s}^{\left(
q\right)  }$ arises from parabolic rescaling. These features remain in play
for an arbitrary quadratic surface of positive Gaussian curvature, and so we
could have used any of these surfaces in place of the paraboloid
$\mathbb{P}^{2}$ in our arguments here.

Thus using (\ref{sq anal}), first with $\rho=2^{-\lambda^{\prime}}$ and then
with $\rho=2^{-\lambda}$, together with (\ref{contrib Case 2}), we obtain that
the contribution $\left\Vert \mathbf{1}_{\Gamma_{\lambda}\left(
\text{\textbf{Case 2}}\right)  }\mathcal{S}_{\operatorname*{Fourier}}%
^{s}f\right\Vert _{L^{q}\left(  \mathbb{R}^{3}\right)  }^{q}$ to the norm
$\left\Vert \mathcal{S}_{\operatorname*{square}}^{s}f\right\Vert
_{L^{q}\left(  \mathbb{R}^{3}\right)  }$ satisfies:
\begin{align}
&  \left\Vert \mathbf{1}_{\Gamma_{\lambda}\left(  \text{\textbf{Case 2}%
}\right)  }\mathcal{S}_{\operatorname*{Fourier}}^{s}f\right\Vert
_{L^{q}\left(  \mathbb{R}^{3}\right)  }\leq C\left(  \#\mathcal{G}%
_{\lambda^{\prime}}\left[  U\right]  \right)  ^{\frac{1}{q}}\left(
2^{-\lambda^{\prime}}\right)  ^{2}\left(  2^{-\lambda^{\prime}}\right)
^{-\frac{4}{q}}A_{s}^{\left(  q\right)  }\label{case 2 est}\\
&  +C\left(  \#\mathcal{G}_{\lambda}\left[  U\right]  \right)  ^{\frac{1}{q}%
}2^{-\frac{3}{q}\lambda}\left(  2^{-\lambda}\right)  ^{2}\left(  2^{-\lambda
}\right)  ^{-\frac{4}{q}}A_{s}^{\left(  q\right)  }\nonumber\\
&  =C2^{\left(  \frac{3}{q}-1\right)  \lambda}A_{s}^{\left(  q\right)
}+C2^{\left(  \frac{3}{q}-2\right)  \lambda}A_{s}^{\left(  q\right)
}\ ,\nonumber
\end{align}
since $\lambda^{\prime}=\frac{1}{2}\lambda-10$ and%
\[
\left(  \#\mathcal{G}_{\frac{1}{2}\lambda-10}\left[  U\right]  \right)
^{\frac{1}{q}}\left(  2^{10-\frac{1}{2}\lambda}\right)  ^{2}\left(
2^{-\frac{1}{2}\lambda}\right)  ^{-\frac{4}{q}}=2^{\frac{1}{q}\lambda
}2^{-\lambda}2^{\frac{2}{q}\lambda}=2^{\left(  \frac{3}{q}-1\right)  \lambda
},
\]
and%
\begin{align*}
&  2^{-\frac{3}{q}\lambda}\left(  \#\mathcal{G}_{\lambda}\left[  U\right]
\right)  ^{\frac{1}{q}}\left(  2^{-\lambda}\right)  ^{2}\left(  2^{-\lambda
}\right)  ^{-\frac{4}{q}}A_{s}^{\left(  q\right)  }\\
&  =2^{-\frac{3}{q}\lambda}2^{\frac{2}{q}\lambda}2^{-2\lambda}2^{\frac{4}%
{q}\lambda}A_{s}^{\left(  q\right)  }=2^{\left(  \frac{3}{q}-2\right)
\lambda}A_{s}^{\left(  q\right)  }\ .
\end{align*}

\end{proof}

\subsubsection{Case 3: Dipole interaction\label{Sub Case 3}}

\begin{proof}
[Proof continued]In \textbf{Case 3} we assume the negation of both
\textbf{Case 1} and \textbf{Case 2}. The failure of clustered interaction in
\textbf{Case 2} implies that there exists $I_{\ast\ast}^{a,s}$ with
$\omega_{I_{\ast\ast}^{a,s}}>2^{-2\lambda}w_{\ast}^{a,s}$ and $\left\vert
c_{I_{\ast\ast}^{a,s}}-c_{I_{\ast}^{a,s}}\right\vert >2^{10-\lambda}$, i.e.%
\begin{align*}
\int_{\mathbb{R}^{3}}\left\vert \mathcal{S}^{s}T_{I_{\ast\ast}^{a,s}}f\left(
z\right)  \right\vert \zeta_{\lambda}\left(  z-a\right)  dz  &  =\omega
_{I_{\ast\ast}^{a,s}}^{a,s}>2^{-2\lambda}\omega_{\ast}^{a,s}=2^{-2\lambda}%
\int_{\mathbb{R}^{3}}\left\vert \mathcal{S}^{s}T_{I_{\ast}^{a}}f\left(
z\right)  \right\vert \zeta_{\lambda}\left(  z-a\right)  dz\ ,\\
&  \operatorname*{dist}\left(  I_{\ast\ast}^{a,s},I_{\ast}^{a,s}\right)
>2^{10-\lambda}\ ,
\end{align*}
The simultaneous failure of separated interaction in \textbf{Case 1} further
implies that%
\begin{align}
\omega_{I}^{a,s}  &  \leq2^{-2\lambda}\omega_{\ast}^{a,s}\text{ if
}\operatorname*{dist}\left(  c_{I},I_{\ast}^{a,s}\cup I_{\ast\ast}%
^{a,s}\right)  >2^{10}2^{-\lambda},\label{simul}\\
&  \fbox{$%
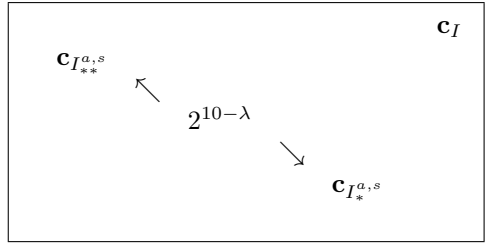
\begin{array}
[c]{cccccccc}
&  &  &  &  &  &  & \mathbf{c}_{I}\\
& \mathbf{c}_{I_{\ast\ast}^{a,s}} &  &  &  &  &  & \\
&  & \nwarrow &  &  &  &  & \\
&  &  & 2^{10-\lambda} &  &  &  & \\
&  &  &  & \searrow &  &  & \\
&  &  &  &  & \mathbf{c}_{I_{\ast}^{a,s}} &  & \\
&  &  &  &  &  &  &
\end{array}
$}\nonumber
\end{align}

In this case we will again use parabolic rescaling since the squares $I$ with
near maximal weight, i.e. $2^{-2\lambda}\omega_{\ast}^{a,s}$, are clustered
within distance $2^{10-\lambda}$ of the squares $I_{\ast}^{a,s}$ and
$I_{\ast\ast}^{a,s}$. Indeed, arguing as in (\ref{contrib Case 2}) and
(\ref{case 2 est}) above, we then have%
\begin{align*}
&  \left\Vert \mathbf{1}_{\Gamma_{\lambda}\left(  \text{\textbf{Case 3}%
}\right)  }\mathcal{S}_{\operatorname*{Fourier}}^{s}f\right\Vert
_{L^{q}\left(  \mathbb{R}^{3}\right)  }\lesssim\left\Vert \sum
_{I:\ \operatorname*{dist}\left(  c_{I},I_{\ast}^{a,s}\cup I_{\ast\ast}%
^{a,s}\right)  >2^{10}2^{-\lambda}}\mathcal{S}_{\operatorname*{Fourier}}%
^{s}T_{I}f\right\Vert _{L^{q}\left(  \mathbb{R}^{3}\right)  }\\
&  \lesssim\left\Vert \sum_{I:\ \operatorname*{dist}\left(  c_{I},I_{\ast
}^{a,s}\right)  \leq2^{10}2^{-\lambda}}\mathcal{S}_{\operatorname*{Fourier}%
}^{s}f_{I}\right\Vert _{L^{q}\left(  \mathbb{R}^{3}\right)  }+\left\Vert
\sum_{I:\ \operatorname*{dist}\left(  c_{I},I_{\ast\ast}^{a,s}\right)
\leq2^{10}2^{-\lambda}}\mathcal{S}_{\operatorname*{Fourier}}^{s}%
T_{I}f\right\Vert _{L^{q}\left(  \mathbb{R}^{3}\right)  }\\
&  \lesssim\left(  2^{-\lambda}\right)  ^{2-\frac{4}{q}}A_{s}^{\left(
q\right)  }\ .
\end{align*}
since the first integral is dominated by%
\[
\#\mathcal{G}_{\lambda}\left[  U\right]  2^{-2\lambda}\omega_{\ast}%
^{a,s}\lesssim\omega_{I_{\ast}^{a}}^{a,s}=\left\Vert \mathcal{S}%
_{\operatorname*{Fourier}}^{s}T_{I_{\ast}^{a}}f\right\Vert _{L^{q}\left(
\mathbb{R}^{3}\right)  }\leq\left(  2^{-\lambda}\right)  ^{2-\frac{4}{q}}%
A_{s}^{\left(  q\right)  },
\]
and the next two integrals are each dominated by%
\[
C\left(  2^{10}2^{-\lambda}\right)  ^{2-\frac{4}{q}}A_{s}^{\left(  q\right)
}.
\]

\end{proof}

\subsection{Completing the proof of the theorem}

\begin{proof}
[Proof continued]So far we have shown that
\begin{align*}
\left\Vert \mathcal{S}_{\operatorname*{Fourier}}^{s}f\right\Vert
_{L^{q}\left(  \mathbb{R}^{3}\right)  }  &  \leq\left\Vert \mathbf{1}%
_{\Gamma_{\lambda}\left(  \text{\textbf{Case 1}}\right)  }\mathcal{S}%
_{\operatorname*{Fourier}}^{s}f\right\Vert _{L^{q}\left(  \mathbb{R}%
^{3}\right)  }+\left\Vert \mathbf{1}_{\Gamma_{\lambda}\left(
\text{\textbf{Case 2}}\right)  }\mathcal{S}_{\operatorname*{Fourier}}%
^{s}f\right\Vert _{L^{q}\left(  \mathbb{R}^{3}\right)  }+\left\Vert
\mathbf{1}_{\Gamma_{\lambda}\left(  \text{\textbf{Case 3}}\right)
}\mathcal{S}_{\operatorname*{Fourier}}^{s}f\right\Vert _{L^{q}\left(
\mathbb{R}^{3}\right)  }\\
&  \lesssim C_{\nu}2^{\left(  3+\frac{15}{q}\right)  \lambda}2^{s\frac
{\varepsilon}{3}}+C2^{\left(  \frac{3}{q}-1\right)  \lambda}A_{s}^{\left(
q\right)  }+C2^{\left(  \frac{3}{q}-2\right)  \lambda}A_{s}^{\left(  q\right)
}+\left(  2^{-\lambda}\right)  ^{2-\frac{4}{q}}A_{s}^{\left(  q\right)  }\\
&  \lesssim C_{\nu}2^{\left(  3+\frac{15}{q}\right)  \lambda}2^{s\frac
{\varepsilon}{3}}+\left[  2^{-\frac{1}{q}\left(  q-3\right)  \lambda
}+2^{-\frac{2}{q}\left(  q-\frac{3}{2}\right)  \lambda}+2^{-\left(  2-\frac
{4}{q}\right)  \lambda}\right]  A_{s}^{\left(  q\right)  }\lesssim2^{\left(
3+\frac{15}{q}\right)  \lambda}2^{s\frac{\varepsilon}{3}}+\frac{1}{2}%
A_{s}^{\left(  q\right)  }\ ,
\end{align*}
if $3<q<4$, and $\lambda$ is sufficiently large, namely $2^{10}2^{-\lambda
}\leq\nu_{q}$ from (\ref{q control}), and%
\begin{equation}
2^{-\left(  q-3\right)  \frac{1}{q}\lambda}+2^{-\left(  q-\frac{3}{2}\right)
\frac{2}{q}\lambda}+2^{-\left(  2-\frac{4}{q}\right)  \lambda}<\frac{1}{2}.
\label{namely}%
\end{equation}
Thus provided (\ref{q control}) and (\ref{namely}) both hold, we obtain%
\[
A_{s}^{\left(  q\right)  }\approx\sup_{s}\sup_{\left\Vert f\right\Vert
_{L^{\infty}}\leq1}\left\Vert \mathcal{S}_{\operatorname*{Fourier}}%
^{s}f\right\Vert _{L^{q}\left(  \mathbb{R}^{3}\right)  }\lesssim2^{\left(
3+\frac{15}{q}\right)  \lambda}2^{s\frac{\varepsilon}{3}}+\frac{1}{2}%
A_{s}^{\left(  q\right)  },
\]
and absorption now yields the inequality,%
\[
A_{s}^{\left(  q\right)  }\lesssim C_{\nu}2^{\left(  3+\frac{15}{q}\right)
\lambda}2^{s\frac{\varepsilon}{3}},\ \ \ \ \ \text{for all }s\in
\mathbb{N}\text{ with }2^{-s}\leq\nu_{q},
\]
This completes the proof of Theorem \ref{Loc lin} with $C_{q}\approx C_{\nu
}2^{\left(  3+\frac{15}{q}\right)  \lambda}$ and $\lambda$ chosen so that both
(\ref{q control}) and (\ref{namely}) hold.
\end{proof}

\begin{remark}
\label{param}The inequality (\ref{namely}) holds if
\[
\left(  1-\frac{3}{q}\right)  \lambda>3\text{ and }\left(  1-\frac{2}%
{q}\right)  \lambda>1\text{, in particular if }\lambda>\frac{3q}{q-3}.
\]
Thus altogether we need to take%
\[
\lambda>\max\left\{  \frac{3q}{q-3},\log_{2}\frac{2^{10}}{\nu_{q}}\right\}  .
\]

\end{remark}

\section{Concluding remarks}

In Theorem \ref{big}, we need not test the trilinear Kakeya inequality
(\ref{tri Kak dual}) over \emph{all} triples $\left(  \mathbb{T}%
_{1},\mathbb{T}_{2},\mathbb{T}_{3}\right)  $ satisfying the $\nu$-disjoint
condition (\ref{nu disjoint}), but only over the corresponding triples of
squares $\left(  U_{1},U_{2},U_{3}\right)  $ satisfying the $\nu$-disjoint
condition (\ref{nu disjoint'}), in which $\ell\left(  U_{1}\right)  \approx1$
and $U_{1}$, $U_{2}$ and $U_{3}$ are roughly lined up in a line, at least in
the case when $U_{1}$ is centered at the origin. Otherwise, the triples
$\left(  U_{1},U_{2},U_{3}\right)  $ in question are such that their sets of
unit normals to $\Phi\left(  U_{1}\right)  ,\Phi\left(  U_{2}\right)
,\Phi\left(  U_{3}\right)  $ are $\nu$-disjoint on $\mathbb{S}^{2}$. Here the
squares $U_{k}$ correspond to the families $\mathbb{T}_{k}$ in the sense that
the orientations of the tubes in $\mathbb{T}_{k}$ lie in the spherical patches
$\Phi\left(  U_{k}\right)  $.

Indeed, suppose for convenience that $U_{1}$ is centered at the origin. Then
given any large positive constant $\Lambda$, it is a simple exercise to first
reduce matters to triples $\left(  U_{1},U_{2},U_{3}\right)  $ of squares
inside $U$ satisfying%
\begin{align}
\ell\left(  U_{2}\right)   &  =\ell\left(  U_{3}\right)  \leq\ell\left(
U_{1}\right)  ,\label{nu disjoint''}\\
\frac{\Lambda}{2}  &  \leq\frac{\operatorname*{dist}\left(  U_{1}%
,U_{2}\right)  }{\ell\left(  U_{1}\right)  },\frac{\operatorname*{dist}\left(
U_{1},U_{3}\right)  }{\ell\left(  U_{1}\right)  },\frac{\operatorname*{dist}%
\left(  U_{2},U_{3}\right)  }{\ell\left(  U_{2}\right)  }\leq2\Lambda
,\nonumber\\
\nu &  \leq\operatorname*{dist}\left(  U_{2},U_{3}\right)  .\nonumber
\end{align}

Then using the parabolic rescaling of Tao, Vargas and Vega \cite{TaVaVe} as in
(\ref{Int est}) above, we may restrict the triples $\left(  U_{1},U_{2}%
,U_{3}\right)  $ in (\ref{nu disjoint''}) to those also satisfying%
\[
\frac{1}{32}\leq\ell\left(  U_{1}\right)  \leq\frac{1}{8},
\]
since rescaling $U_{1}$ to up a larger set only increases the left hand side
of (\ref{single tri Four}). As was done earlier, a rescaling down by a factor
of $\frac{1}{10}$ at the origin, repositions the triple so that a rescaling up
by a factor approximately $\frac{1}{\ell\left(  U_{1}\right)  }$ from the
center of $U_{1}$, keeps the triple within $U$.

Using the transverse trilinear theorem of Bennett, Carbery and Tao
\cite[Theorem 1.15]{BeCaTa}, and continuing to suppose that $U_{1}$ is
centered at the origin, we may further restrict the triples $\left(
U_{1},U_{2},U_{3}\right)  $ in (\ref{nu disjoint''}) to those satisfying%
\[
U_{1}\cap2\mathbb{K}\left(  U_{2},U_{3}\right)  \mathbb{\neq\emptyset},
\]
where $\mathbb{K}\left(  U_{2},U_{3}\right)  $ is the thinnest cone centered
at $\frac{c_{U_{2}}+c_{U_{3}}}{2}$ that contains both $U_{2}$ and $U_{3}$.
This is because any $\nu$\emph{-disjoint} triple $\left(  U_{1},U_{2}%
,U_{3}\right)  $, in which $2\mathbb{K}$ is disjoint from $U_{1}$, is $\nu
$\emph{-transverse} in the sense of \cite{BeCaTa}, and so their inequality applies.

The resulting $\nu$-disjoint configuration of squares over which we must test,
with $\Lambda\approx3$, resembles,%

\[
\left\{
\begin{array}
[c]{cccccccc}%
\fbox{$%
\begin{array}
[c]{cc}
& \\
&
\end{array}
$}U_{1} &  &  &  &  &  &  & \\
&  &  &  &  &  &  & \\
&  &  &  &  &  &  & \\
&  &  &  &  &  &  & \\
&  &  &  &  &  &  & \\
&  &  &  &  & \fbox{}U_{2} &  & \\
&  &  &  &  &  & \cdot & \\
&  &  &  &  &  &  & \fbox{}U_{3}%
\end{array}
\right\}  ,
\]
where the dot $\cdot$ represents the point $\frac{c_{U_{2}}+c_{U_{3}}}{2}$
midway between the centers of $U_{2}$ and $U_{3}$, and $U_{1}$ is centered at
the origin.

Note that we need \emph{not} test over configurations of the form,%
\[
\left\{
\begin{array}
[c]{cccccccc}%
\fbox{$%
\begin{array}
[c]{cc}
& \\
&
\end{array}
$}U_{1} &  &  &  &  &  &  & \\
&  &  &  &  &  &  & \\
&  &  &  &  &  &  & \\
&  &  &  &  &  &  & \\
&  &  &  &  &  &  & \\
&  &  &  &  &  &  & \fbox{}U_{2}\\
&  &  &  &  &  & \cdot & \\
&  &  &  &  & \fbox{}U_{3} &  &
\end{array}
\right\}  ,
\]
which are $\nu$-transverse, since any triangle with one vertex in each square
has area bounded below by $c\nu$.

\end{document}